\documentclass[12pt]{amsart}

\usepackage{amsmath,amsfonts,latexsym,graphicx,amssymb,url}
\usepackage{hyperref,pdfsync}
\usepackage{amsmath,txfonts,pifont,bbding,pxfonts,manfnt}
\usepackage[active]{srcltx}
\usepackage{wasysym,pstricks}

\numberwithin{equation}{section}

\newcommand{\diam}{\text{{\rm diam}}}

\newcommand{\R}{{\mathbb R}} %%reals

\def \H {\mathcal{H}}
\def \bd {{\emph{bdry}}}

\newtheorem{Theorem}{Theorem}[section]

\newtheorem{Lemma}[Theorem]{Lemma}
\newtheorem{Definition}[Theorem]{Definition}

\begin{document}
\title{$C^{1,\alpha}$ estimates for the parallel refractor}
%\title{H\"older estimates for gradients of the parallel refractor}
%\title{Regularity of Parallel Refractor: the case $\kappa > 1$}
\author[F. Abedin, C. E. Guti\'errez,
and G. Tralli]{Farhan Abedin, Cristian E. Guti\'errez, and Giulio Tralli}
\date{\today}
\thanks{C. E. G was partially supported
by NSF grant DMS--1201401. G. T. wishes to thank the Department of Mathematics of Temple University for the kind hospitality during his visit during the Spring of 2015, when this work started.}
\maketitle

\begin{abstract}
We consider the parallel refractor problem when the planar radiating source lies in a medium having higher refractive index than the medium in which the target is located. 
We prove local $C^{1,\alpha}$ estimates for parallel refractors under suitable geometric assumptions on the source and target, and under local regularity hypotheses on the target set.
We also discuss existence of refractors under energy conservation assumptions.
\end{abstract}

\setcounter{equation}{0}
\section{Introduction}

Suppose we have a domain $\Omega\subset \R^n$ and 
a domain $\Sigma$ contained in an $n$ dimensional surface in $\R^{n+1}$; here,
$\Omega$ denotes the extended source, and $\Sigma$ denotes the target domain, receiver, or screen to be illuminated.
%(for the practical applications,
%one can think that $n=3$).
Let $n_1$ and $n_2$ be the indices of refraction
of two homogeneous and isotropic media I and II, respectively.
Suppose from the extended source $\Omega$, surrounded by medium I,
radiation emanates in the vertical direction $e_{n+1}$ with intensity $f(x)$
for $x\in \Omega$, and the target $\Sigma$ is surrounded by medium II.
That is, all emanating rays from $\Omega$ are collimated.
A {\it parallel refractor} is an optical surface
$\mathcal R$, 
interface between media I and II,
such that all rays refracted by $\mathcal R$ into medium II
are received at the surface $\Sigma$ with prescribed radiation intensity $\sigma(p)$
at each point $p\in \Sigma$.
%Of course some conditions on the relative position of $\Sigma$ and and $\Omega$ are needed so rays can be refracted to $\Sigma$,
%see conditions (A) and (B) below.
Assuming no loss of energy in this process, we have the conservation of energy equation
$\int_{\Omega}f(x)\,dx=\int_{\Sigma}\sigma(p)\,dp$.

When medium II is denser than medium I (i.e. $n_1<n_2$), $C^{1,\alpha}$ estimates are proved in \cite{gutierrez-tournier:regularityparallelrefractor}, and the existence of refractors is proved in \cite{gutierrez-tournier:parallelrefractor}. The purpose of this paper is to consider the case when $n_1>n_2$. This has interest in the applications to lens design since lenses are typically made of a material having refractive index larger than the surrounding medium. 
In fact, if the material around the source is cut out with a plane parallel to the source, then the lens sandwiched  between that plane and the constructed refractor surface will perform the desired refracting job.
When $n_1>n_2$ the geometry of the refractors is different than when $n_1<n_2$; in fact, the geometry is determined by hyperboloids instead of ellipsoids. 
In addition, in the case $n_1>n_2$, total internal reflection can occur and one needs additional geometric conditions on the relative configuration between the source and the target so that the target is reachable by the refracted rays.
To obtain existence and regularity of refractors when $n_1>n_2$, the use of hyperboloids requires non-trivial changes in some of the arguments used in \cite{gutierrez-tournier:regularityparallelrefractor} when $n_1<n_2$. The main differences are in the set up of the problem, in the arguments to obtain global support from local support, Section \ref{sec:localtoglobal}, and in the proof of existence. 
Our results are local; that is, we only need to assume local conditions in a neighborhood of a point in the extended source and the target.
The main result of the paper is Theorem \ref{thm:Holdercontinuityresult} where $C^{1,\alpha}$ estimates are proved. 
We remark that most results do not involve the energy distribution given in the source and target,
and conservation of energy is only used to prove existence in Theorem \ref{thm:existencediscretecase}.
For instance, the fact that local refractors are global, Theorem \ref{localtoglobal}, just follows from the geometric assumptions in Section \ref{asstar}; see condition \eqref{eq:AWcondition}.
In addition, Theorem \ref{lowerboundtheorem} only requires geometric assumptions.
Properties of the target measure are necessary only to obtain the H\"older estimates, Theorem \ref{thm:Holdercontinuityresult}. 
Our results are structural, in the sense that they only depend on the geometric conditions assumed and do not depend on the smoothness of the measures given in the source and target.

Problems of refraction have generated interest recently for the applications to design free form lenses and also for the various mathematical tools developed to solve them. 
For example, the far field point source refractor problem is solved in \cite{gutierrez-huang:farfieldrefractor} using mass transport. The near field point source refractor problem is considered in \cite{gutierrez-huang:nearfielrefractorermanno6thbirth} and \cite{gutierrez-huang:nearfieldrefractor}. More general models taking into account losses due to internal reflection are in \cite{gutierrez-mawi:refractorwithlossofenergy}. 
Numerical methods have been developed in \cite{Brix15} and \cite{Caffarelli-Oliker:weakantenna} for the actual calculation of reflectors, and recently in \cite{deleo-gutierrez-mawi:farfieldnumerical} for the numerical design of far field point source refractors. A significant amount of work has also been done to obtain results on the regularity of reflectors and refractors \cite{caffarelli-gutierrez-huang:antennaannals, loeper2011:regularityonthesphere, karakhanyan-wang:nearfieldreflector, karakhanyan2014:collimatedreflector, karakhanyan2016parallellighting, guillen2015pointwise}.

The organization of the paper is as follows. Section \ref{sec:definitionsandpreliminaries} contains results concerning estimates of hyperboloids of revolution. The precise definition of refractor is in Section \ref{condi}, and the structural assumptions on the target that avoid total reflection are in Section \ref{subsec:structuralassumptionsontarget}.
The derivatives estimates needed for hyperboloids are in Section \ref{subsec:derivativeestimates}.
Section \ref{asstar} contains assumptions on the target modeled on the conditions introduced by 
Loeper in the seminal work \cite[Proposition 5.1]{loeper:actapaper}. 
In Section \ref{sec:localtoglobal}, using the geometry of the hyperboloids, we prove that if a hyperboloid supports a parallel refractor 
locally, then it supports the refractor globally provided the target satisfies the local condition \eqref{eq:AWcondition}.
This resembles the condition (A3) of Ma, Trudinger and Wang \cite{MaTrudingerWang:regularityofpotentials} introduced in the context of optimal mass transport.
The main results are in Section \ref{main}, in particular, Section \ref{subsec:holderestimatesfromgrowthconditions} contains the proof of the H\"older estimates.
Finally in the Appendix, Section \ref{app}, we discuss and establish the existence of refractors satisfying the energy conservation condition \eqref{eq:energyconservation}.

\setcounter{equation}{0}
\section{Definitions and Preliminary Results}\label{sec:definitionsandpreliminaries}

We briefly review the process of refraction.
Points in $\R^{n+1}$ will be denoted by $X=(x,x_{n+1})$. We consider parallel rays traveling in the unit direction $e_{n+1}$. Let $T$ be a hyperplane with outward pointing unit normal $N$ and $X\in T$. We assume that medium $I$ is located in the region below $T$ and media $II$ in the region above $T$. In such a scenario, a ray of light emanated from $\Omega$ in the direction $e_{n+1}$ strikes $T$ at $X$ and, by Snell's Law of Refraction, gets refracted in the unit direction 
$$
\Lambda=\kappa\, e_{n+1} +\delta N, \quad \text{with}\quad 
\delta=
%-\kappa\,e_{n+1}\cdot N  +\sqrt{\kappa^{2}\,(e_{n+1}\cdot N)^{2}+1-\kappa^2},
-\kappa\,e_{n+1}\cdot N  +\sqrt{1+\kappa^{2}\,\left((e_{n+1}\cdot N)^{2}-1\right)},$$ 
where $\delta<0$ since $\kappa>1$. The refracted ray is $X+s\Lambda$, for $s>0$; see Figure \ref{fig:snell}.
In particular, if $v\in \R^{n}$ and the hyperplane $T$ is so that the unit upper normal $N=\dfrac{(-v,1)}{\sqrt{1+|v|^2}}$, then $\delta=\dfrac{-\kappa+\sqrt{1-(\kappa^2-1)|v|^2}}{\sqrt{1+|v|^2}}$ and the refracted unit direction is 
\begin{equation}\label{eq:defofLambdav}
\Lambda (v):=\left(-\dfrac{\delta}{\sqrt{1+|v|^2}}\, v, \dfrac{\delta}{\sqrt{1+|v|^2}}+\kappa \right):=(Q(v)v,-Q(v)+\kappa),
\end{equation}
with $Q(v)^{2}|v|^{2}+(Q(v)-\kappa)^{2}=1$. With this notation we have $Q>0$.

Since medium $I$ is more dense than medium $II$, total internal reflection can occur, \cite[Sect. 1.5.4]{book:born-wolf}. To avoid this we assume $e_{n+1}\cdot \Lambda(v)\geq n_2/n_1$, or equivalently, $e_{n+1}\cdot N\geq \sqrt{1-\kappa^{-2}}$; see \cite[Lemma 2.1]{gutierrez-huang:farfieldrefractor} where $n_1$ and $n_2$ are reversed. 

\begin{figure}
  \centering
\includegraphics[width=2in]{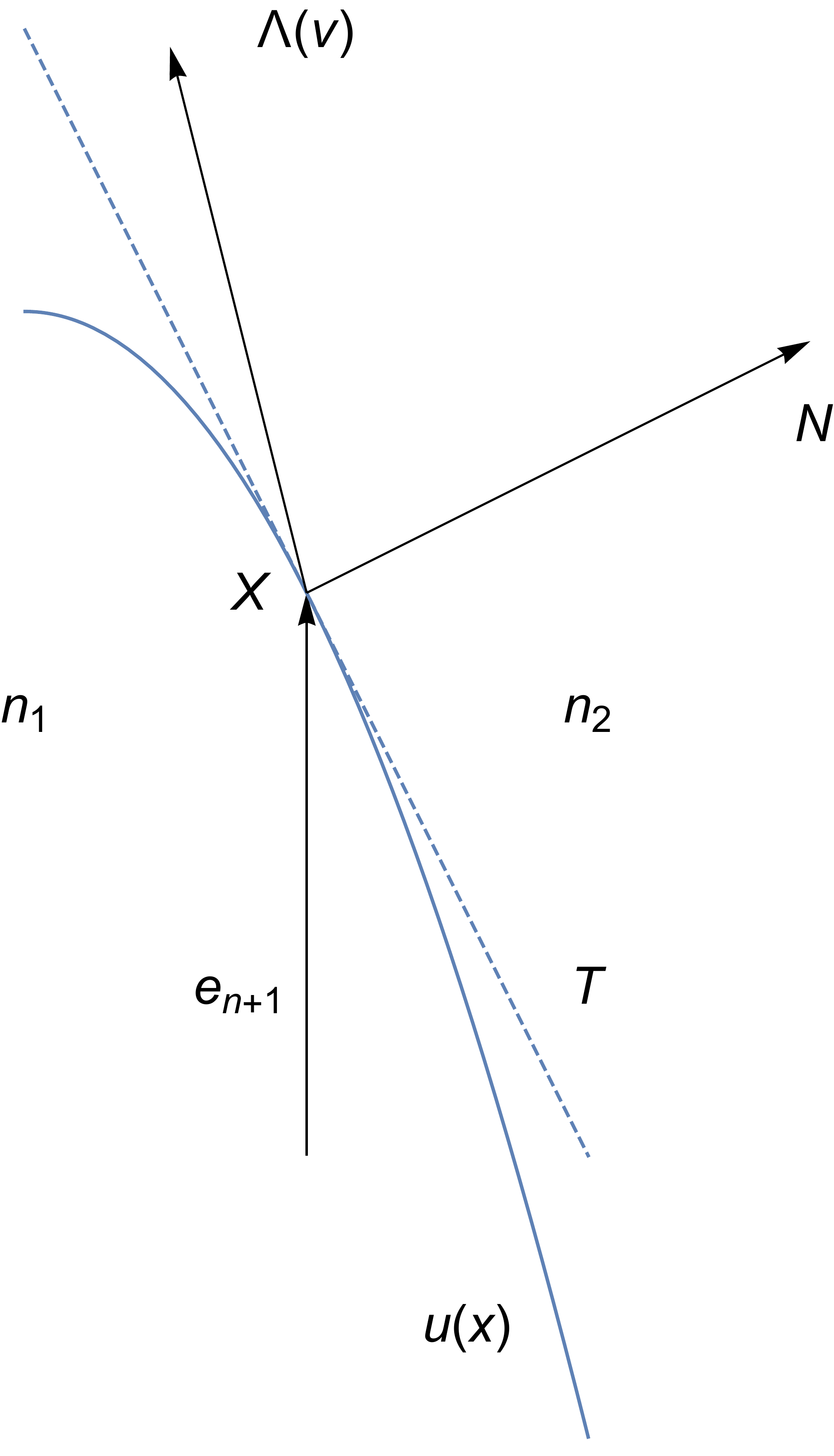}
\caption{Snell $n_1>n_2$}
\label{fig:snell}
\end{figure}

%\textcolor{red}{?}Do we want a Figure here? YES

\subsection{Hyperboloids}\label{fixing}

Fix $b>0$. %A hyperboloid of two sheets in $\R^{n+1}$ with foci at $(0,0)$ and $(0,2b\kappa/(\kappa^2-1))$ has equation
%\[
%\dfrac{\left(x_{n+1}-\dfrac{\kappa\,b}{\kappa^2-1} \right)^2}{\left( \dfrac{b}{\kappa^2-1}\right)^2}-\dfrac{|x|^2}{\left(\dfrac{b}{\sqrt{\kappa^2-1}}\right)^2}=1.
%\]
A two-sheeted hyperboloid in $\R^{n+1}$ with upper focus at $Y=(y,y_{n+1})$ and lower focus at $Y'=\left(y,y_{n+1}-\dfrac{2\kappa\,b}{\kappa^2-1} \right)$ has equation
\[
\dfrac{\left(x_{n+1}-\left(y_{n+1}+\dfrac{\kappa\,b}{\kappa^2-1}\right) \right)^2}{\left( \dfrac{b}{\kappa^2-1}\right)^2}-\dfrac{|x-y|^2}{\left(\dfrac{b}{\sqrt{\kappa^2-1}}\right)^2}=1.
\]
The semi-axis with direction $y_{n+1}$ is $\dfrac{b}{\kappa^2-1}$, the semi-axis with direction $y$ is $\dfrac{b}{\sqrt{\kappa^2-1}}$, and the center of symmetry is the point 
$\left(y,y_{n+1}-\dfrac{\kappa\,b}{\kappa^2-1} \right)$. Moreover, the upper vertex is 
$\left(y,y_{n+1}-\dfrac{b}{\kappa+1} \right)$, and the lower vertex is 
$\left(y,y_{n+1}-\dfrac{b}{\kappa-1} \right)$. Hence the distance between the foci is
\[
2c=\dfrac{2\kappa\,b}{\kappa^2-1},
\] 
and the distance between the vertices is 
\[
2a=\dfrac{2b}{\kappa^2-1}.
\] 
By definition, the eccentricity is $\dfrac{c}{a}$, and so the eccentricity equals $\kappa$.
The lower sheet (facing downwards) of the hyperboloid is given by
\begin{equation}\label{polareqn}
\H(Y,b) := \left\{X=(x,x_{n+1}) \in \mathbb{R}^{n+1} : \kappa (y_{n+1} - x_{n+1}) - |X-Y| = b \right\},
\end{equation}
%\bigskip
%\noindent 
which can be written as the graph of the function
\begin{equation}\label{hyperboloid}
\phi(x) := \phi_{Y,b}(x):=y_{n+1}-\dfrac{\kappa\,b}{\kappa^2-1}-\sqrt{\left(\dfrac{b}{\kappa^2-1}\right)^2 + \dfrac{|x-y|^2}{\kappa^2-1}}.
\end{equation}
%\bigskip
%\noindent 
%(since $\phi(x) \rightarrow -\infty$ as $|x| \rightarrow \infty$).
%
%We remark that
%\[
%\phi_{Y,b}(x)\leq \phi_{Y,b'}(x)\qquad \text{when $b'\leq b$},
%\]
%and the maximum value of $\phi_{Y,b}$ is attained at $y$ and is equal to
%\begin{equation}\label{eq:maximumofphiYb}
%\max_x \phi_{Y,b}(x)=y_{n+1}-\dfrac{b}{\kappa-1}.
%\end{equation}

Suppose the region above $\H(Y,b)$ has refractive index $n_2$, and the region below $\H(Y,b)$ has refractive index $n_1$, with $\kappa=n_1/n_2>1$. Then we have from \cite[Section 2.2]{gutierrez-huang:farfieldrefractor} and the reversibility of optical paths that each ray with direction $e_{n+1}$ striking from below the graph of $\phi$ at the point $X=(x,\phi(x))$ is refracted into a ray passing through the upper focus $Y$; see Figure \ref{fig:refractioninhyperboloid}. Therefore, $Y$ lies along the ray $X+s\Lambda(v)$ with $v=D\phi(x)$, with $\Lambda(v)$ given by \eqref{eq:defofLambdav}. Conversely, 
\begin{eqnarray}\label{Fact}
&\,&\mbox{if $X=(x,\phi(x))$ and the focus $Y$ can be written as $Y=X+s\,\Lambda(v)$}\\
&\,&\mbox{for some $s>0$ and $v\in \R^{n}$, then $v=D\phi(x)$;}\nonumber
\end{eqnarray}
a fact that will be used on multiple occasions throughout this paper.

\begin{figure}
  \centering
\includegraphics[width=2.5in]{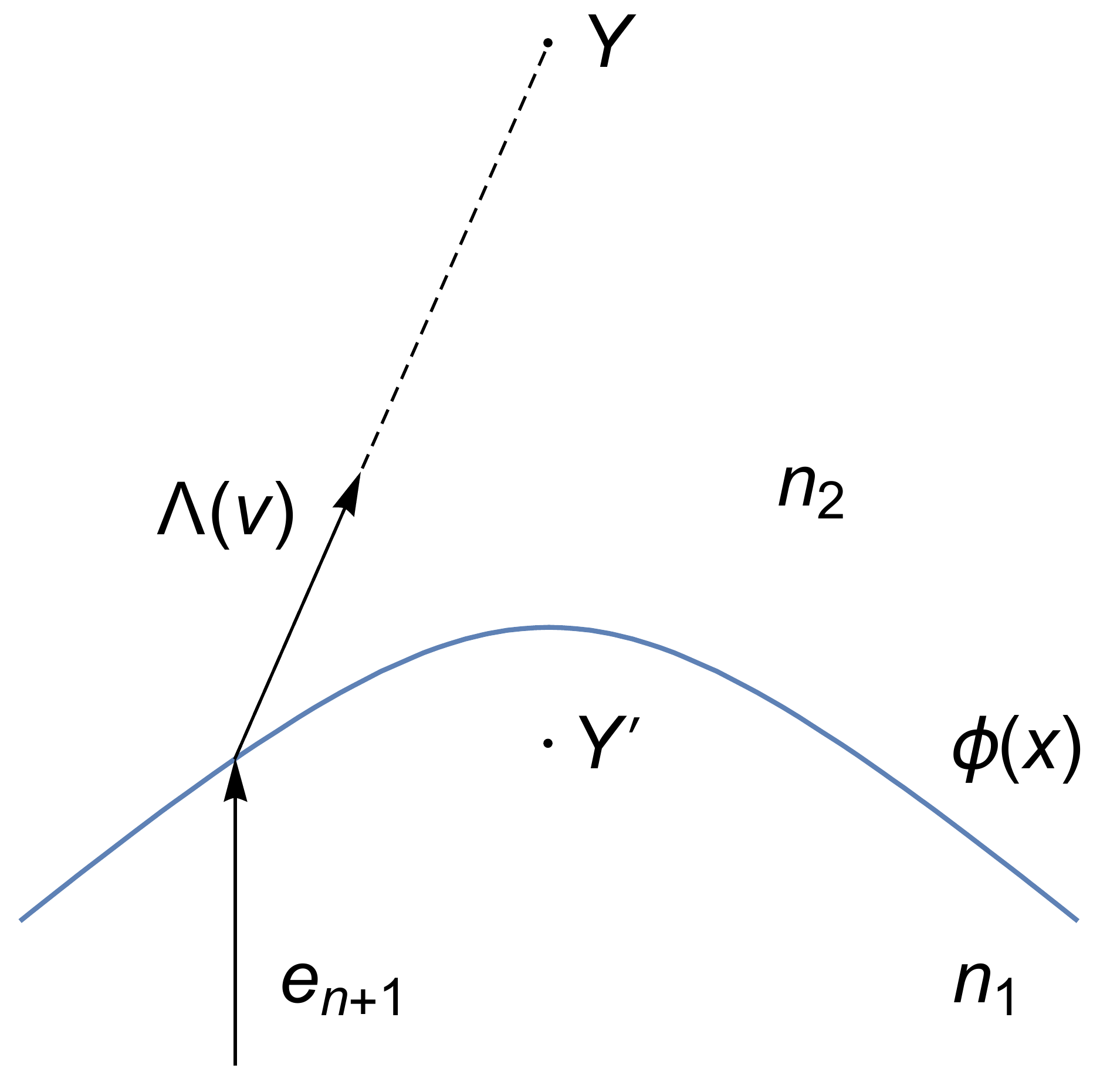}
\caption{Refraction in the hyperboloid}
\label{fig:refractioninhyperboloid}
\end{figure}

\vskip .02in

Given $X,Y\in \R^{n+1}$, let us define 
\begin{equation}\label{eq:definitionofc(X,Y)}
c(X,Y):=\kappa(y_{n+1}-x_{n+1})-|X-Y|.
\end{equation}
%Of course we have 
%\begin{equation}\label{eq:estimatesforc(X,Y)}
%c(X,Y)\leq (\kappa-1)\,|X-Y|.
%\end{equation}
%Given $X_0,Y\in \R^{n+1}$ with 
If $c(X_0, Y)>0$, then $\H(Y,c(X_0,Y))$ is the unique lower sheet of a two-sheeted hyperboloid with upper focus at $Y$ passing through $X_0$, and it is thus described by
\begin{equation}\label{eq:definitionofphi}
\phi(x,Y,X_0)= y_{n+1}-\frac{\kappa\,c(X_0,Y)}{\kappa^{2}-1}-\sqrt{\frac{c(X_0,Y)^{2}}{(\kappa^{2}-1)^{2}}+\frac{|x-y|^{2}}{\kappa^{2}-1}}.
\end{equation}
Notice that $\phi(x,Y,X_0)=\phi_{Y,c(X_0,y)}(x)$.

\subsection{Definition of refractor}\label{condi}

%Let us consider $\Omega$ a bounded domain in $\R^n=\R^n\times\{0\}$, which is the source of light surrounded by medium $n_1$. Let also $\Sigma \subset \mathbb{R}^{n+1}$ be a compact hypersurface such that $\Sigma \subset \{x_{n+1}>0\}$.
%\mathbb{R}^n \times [\tau_1, \tau_1+\omega]$ for positive constants $\tau_1, \omega$. Put $\tau_2=\tau_1+\omega$. %Let $\pi(\Sigma)$ be the the projection of $\Sigma$ onto $\mathbb{R}^n$ and let $\Delta = \mathrm{diam}(\Omega \cup \pi(\Sigma))$. 
% The surface $\Sigma\subset \R^{n+1}$ is the target we want to illuminate, and it is surrounded by medium $n_2$. We seek for a refractor
%The refractor surface $\mathcal{R}$, which will be given by the graph of a function $u$, is expected to be in a cylindrical region  
%$$\mathcal{C}_{\Omega}= \Omega \times [0,\tau_0],$$
%for some $0<\tau_0<\tau_1$.
%
%\vskip 0.6cm
%
%We are going to assume some compatibility conditions involving $\kappa, \tau_1,\tau_0, \omega$, and the relative positions of $\Omega, \Sigma$. At the end of this subsection we will state such conditions, together with the precise definition of refractor surface.
%
%\vskip 0.6cm

We are given a source domain $\Omega\subset \R^n=\R^n\times\{0\}$ surrounded by medium $n_1$ and a target $\Sigma$, a compact hypersurface in $\R^{n+1}_+=\{x_{n+1}>0\}$, surrounded by medium $n_2$, with $n_2<n_1$. 
Informally, a parallel refractor from $\Omega$ to $\Sigma$ is the graph of a function $u$ defined on $\Omega$ that refracts all vertical rays emanating from $\Omega$ into $\Sigma$. %For the definition we use hyperboloids supporting $u$ from below.
The hyperboloid $\mathcal H(Y,b)$ is said to support $u$ at the point $(x_0, u(x_0)), x_0 \in \Omega$, if there exist $b>0$ and $Y\in \Sigma$ such that $u(x)\geq \phi_{Y,b}(x)$ with equality at $x=x_0$. We will show that the existence of supporting hyperboloids depends on the relative positions between $\Omega$ and $\Sigma$; this will lead to a precise notion of refractor given in Definition \ref{def:refractor}.

Also from physical reasons, the refracting surface given by $u$ must be above the source $\Omega$: $u$ has thus to be positive in $\Omega$. 
%The refractor $\{(x,u(x))\}$ is by definition supported by hyperboloids from below, and is located by construction above the source. 
This means that the supporting hyperboloids must satisfy
\[\phi_{Y,b}(x)>0\qquad \text{for all $x\in \Omega$ and for all $Y\in \Sigma$,}
\]
which immediately imposes a condition on $b$. In fact, first notice that from
\eqref{hyperboloid} we have
\begin{equation}\label{eq:maximumofphiYb}
\max_x \phi_{Y,b}(x)=\phi_{Y,b}(y)=y_{n+1}-\dfrac{b}{\kappa-1}.
\end{equation}
If $\phi_{Y,b}(\bar{x})>0$ at some $\bar{x}\in \Omega$, then we have $\phi_{Y,b}(y)>0$ that is
%we have
%\[
%y_{n+1}-\dfrac{b}{\kappa-1}\geq \phi_{Y,b}(x)>0
%\]
%which implies
\begin{equation}\label{eq:firstconditionforb}
0<b<(\kappa-1)\,y_{n+1}.
\end{equation}
Fix $Y\in \Sigma$ and $b$ satisfying \eqref{eq:firstconditionforb}. By calculation we get that
\[
\{x\in \R^n:\phi_{Y,b}(x)> 0\}
=
B\left(y,\sqrt{\left(b-\kappa\,y_{n+1} \right)^2-y_{n+1}^2} \right).
\]
%To define the refractor $u$ in $\Omega$ we require that all supporting hyperboloids are positive in $\Omega$,that is, we need
Since we need all the $\phi_{Y,b}$'s to be positive in $\Omega$, we want
\begin{equation}\label{eq:inclusionOmegaintoball}
\Omega\subset B\left(y,\sqrt{\left(b-\kappa\,y_{n+1} \right)^2-y_{n+1}^2} \right)\qquad \mbox{for all }Y=(y,y_{n+1})\in \Sigma.
\end{equation}
Notice that \eqref{eq:firstconditionforb} implies that the quantity inside the last square root is positive. Fixing $Y=(y,y_{n+1}) \in \Sigma$ and letting $\Delta_y=\diam\left(\Omega\cup \{y\} \right)$, \eqref{eq:inclusionOmegaintoball} is equivalent to
\begin{equation}\label{eq:conditionforDelta}
\Delta_y\leq \sqrt{\left(b-\kappa\,y_{n+1} \right)^2-y_{n+1}^2}
\end{equation}
which squaring imposes a condition on $b$, i.e.
\begin{equation}\label{eq:quadraticequationforb}
b^2 -2\,\kappa\,b\,y_{n+1}+(\kappa^2-1)\,y_{n+1}^2-\Delta_y^2\geq 0.
\end{equation}
The corresponding quadratic equation in $b$ has roots
\[
r_\pm=\kappa\,y_{n+1}\pm \sqrt{y_{n+1}^2+\Delta_y^2}.
\]
First observe that $r_->0$. Because there is $x_0\in \overline{\Omega}$ such that $\Delta_y=|x_0-y|$ and since $\phi_{Y,b}(x_0)\geq 0$ we obtain 
\[
y_{n+1}>\dfrac{\Delta_y}{\sqrt{\kappa^2-1}}
\]
which is equivalent to $r_->0$.
So to have the inclusion \eqref{eq:inclusionOmegaintoball} we must have from \eqref{eq:quadraticequationforb} that
\[
0<b<r_-\qquad \text{or} \qquad r_+<b.
\]
But from \eqref{eq:firstconditionforb} it is easy to see that $r_+<b$ is impossible.
So to have the inclusion \eqref{eq:inclusionOmegaintoball} we must have
\[
0<b<\min\{(\kappa-1)\,y_{n+1},r_-\}=r_-,
\]
%and it is easy to see that $m=r_-$.
that is,
%Therefore to have the inclusion \eqref{eq:inclusionOmegaintoball} we must have
\[
0<b<\kappa\,y_{n+1}- \sqrt{y_{n+1}^2+\Delta_y^2}.
\]
We now choose a uniform bound for $b$ in $y$. Let
\begin{equation}\label{def:bigDelta}
\Delta=\max_{y\in \pi(\Sigma)}\Delta_y
\end{equation}
where $\pi(\Sigma)$ is the projection onto $\R^n$ of the target $\Sigma$.
We require that
\begin{equation}\label{eq:uniformconditionforb}
0<b<\kappa\,y_{n+1}- \sqrt{y_{n+1}^2+\Delta^2}.
\end{equation}
For this to be well defined we need the right hand side to be positive, which means
\[
y_{n+1}>\dfrac{\Delta}{\sqrt{\kappa^2-1}}\qquad \text{for all $(y,y_{n+1})\in \Sigma$.}
\]
So we assume that the target satisfies the condition
%\begin{equation}\label{eq:lowerboundforyn+1}
%y_{n+1}\geq \tau_1>\dfrac{\Delta}{\sqrt{\kappa^2-1}}\qquad \text{for all $(y,y_{n+1})\in \Sigma$.}
%\end{equation}
\begin{equation}\label{eq:lowerboundforyn+1}
\inf_{Y=(y,y_{n+1})\in\Sigma}{y_{n+1}}>\dfrac{\Delta}{\sqrt{\kappa^2-1}.}
\end{equation}
We can now define refractor as follows.

\begin{Definition}\label{def:refractor}
Let $\Delta$ be defined by \eqref{def:bigDelta} and assume \eqref{eq:lowerboundforyn+1}.
The function $u:\Omega\to (0,+\infty)$ is a refractor from $\Omega$ to $\Sigma$ if for each $x_0\in \Omega$ there exists $Y=(y,y_{n+1})\in \Sigma$ and $0<b<\kappa\,y_{n+1}- \sqrt{y_{n+1}^2+\Delta^2}$ such that $u(x)\geq \phi_{Y,b}(x)$ for all $x\in \Omega$ with equality at $x=x_0$.
\end{Definition}

\subsection{Structural Assumptions on the Target}\label{subsec:structuralassumptionsontarget}

From here onwards, we will assume that $\Omega$ is convex, and $\Sigma\subset \mathbb{R}^n \times [\tau_1, \tau_1+\omega]$ for positive constants $\tau_1, \omega$; with $\tau_1$ to be chosen in a moment. Let $\tau_2=\tau_1+\omega$. By \eqref{eq:lowerboundforyn+1}, we require $\tau_1>\dfrac{\Delta}{\sqrt{\kappa^2-1}}$. We assume that the graph of our refractor $u$ is contained in the cylindrical region  $$\mathcal{C}_{\Omega}= \Omega \times [0,\tau_0]$$
for some $0<\tau_0<\tau_1$, that is, $u: \Omega \rightarrow [0,\tau_0]$; see Figure \ref{fig:figurefarfield} in the Appendix. Let us also suppose the following compatibility condition:
\begin{equation}\label{compatibilitycondition}
\tau_1 \geq \max \left\{\kappa \tau_0, \tau_0 + \frac{\kappa \Delta}{\kappa - 1} \right\}.
\end{equation}

In Appendix \ref{app}, we show under this configuration the existence of such a refractor. More precisely, we will prove that, for any $\kappa>1$, $\Delta>0$, $\omega>0$, one can choose $\tau_1>0$, sufficiently large, and $0<\tau_0<\tau_1$, both depending only on $\kappa, \Delta$ and $\omega$, such that \eqref{compatibilitycondition} holds and there exists a refractor $u$ in the sense of Definition \ref{def:refractor}; see Theorem \ref{thm:existencediscretecase} and the comment afterwards. In addition, the refractor constructed there satisfies the energy condition \eqref{eq:energyconservation}. 

Since $n_1>n_2$, total reflection can occur, \cite[Section 1.5.4]{book:born-wolf}. To avoid this, we require that the target $\Sigma$ satisfies 
\begin{equation}\label{eq:toavoidtotalreflection}
e_{n+1}\cdot \dfrac{Y-X}{|Y-X|}\geq \dfrac{n_2}{n_1}\qquad \forall X\in \mathcal C_\Omega,Y\in \Sigma.
\end{equation}
This means the following: if for each $X\in C_\Omega$ we consider the upward cones $\mathcal C_X$ with vertex at $X$ and opening $\phi:=\arccos (n_2/n_1)$, then \eqref{eq:toavoidtotalreflection} is equivalent to say that $\Sigma\subset \cap_{X\in \mathcal C_\Omega}\mathcal C_X$. If $X=(x,x_{n+1})\in \mathcal C_\Omega$, then $0\leq x_{n+1}\leq \tau_0$, and since the cones are vertical, we have $\mathcal C_{(x,\tau_0)}\subset \mathcal C_X$. Therefore 
\[
\cap_{x\in \Omega}\mathcal C_{(x,\tau_0)}\subset \cap_{X\in \mathcal C_\Omega}\mathcal C_X.
\]
If we assume $\Sigma\subset \cap_{x\in \Omega}\mathcal C_{(x,\tau_0)}:=S$, then \eqref{eq:toavoidtotalreflection} holds choosing $\Omega$ appropriately.
For example, if $\Omega=B_r(x_0)$ and we look at the cones $\mathcal C_{(x,\tau_0)}$ with $x\in B_r(x_0)$, we see that the set $S$ is a cone with the same opening $\phi$ and vertex at the point $(x_0,\tau_0+ h)$, with $h=r\tan (\pi/2 -\phi)=r\dfrac{\cos \phi}{\sin \phi}=r\dfrac{1}{\sqrt{\kappa^2-1}}$. If we choose $r$ sufficiently small such that $\tau_0+h=\tau_0+r\dfrac{1}{\sqrt{\kappa^2-1}}<\tau_1$, then $S$ intersected with the slab $\mathbb{R}^n \times [\tau_1, \tau_1+\omega]$ is non-empty. Therefore, taking a target $\Sigma\subset S\cap \mathbb{R}^n \times [\tau_1, \tau_1+\omega]$, there is no total reflection, that is, condition
\eqref{eq:toavoidtotalreflection} holds and $\Sigma$ satisfies the previous structural assumptions.

%\textcolor{red}{conservation of energy?}

\subsection{Derivative Estimates for Hyperboloids}\label{subsec:derivativeestimates}

Let us first observe that hyperboloids are uniformly Lipschitz hypersurfaces. Indeed, a direct calculation shows that
%\begin{equation}\label{derphi}
$$D_x\phi(x,Y,X_0)=\dfrac{y-x}{\sqrt{c(X_0,Y)^2+(\kappa^2-1)|x-y|^{2}}}.$$
%\end{equation}
%\noindent 1. Derivatives in $x$ for $i = 1, \ldots, n$.
%
%\begin{align*}
%\frac{\partial \phi}{\partial x_i}(x, Y, X_0) & = -\frac{1}{\kappa^2 - 1} \left(\frac{x_i - y_i}{\sqrt{\frac{c(X_0,Y)^2}{(\kappa^2 - 1)^2} + \frac{|x-y|^2}{\kappa^2 - 1}}} \right) \\
%& = \frac{y_i - x_i}{\sqrt{c(X_0,Y)^2 + (\kappa^2 - 1)|x-y|^2}}.
%\end{align*}
Therefore,
\begin{equation}\label{derivative}
\left|D_x\phi(x,Y,X_0)\right|\leq\min\left\{\dfrac{|y-x|}{c(X_0,Y)},\,\frac{1}{\sqrt{\kappa^2 - 1}} \right\}\leq 
\dfrac{1}{\sqrt{\kappa^2-1}}.
\end{equation}
%\begin{equation}\label{derivative}
%\displaystyle{\left|\frac{\partial \phi}{\partial x_i}(x, Y, X_0) \right| \leq \min\left\{\frac{|x_i - y_i|}{c(X_0,Y)}, \frac{1}{\sqrt{\kappa^2 - 1}} \right\}}.
%\end{equation}
%\bigskip
Hence, by the definition of refractor, we conclude that 
\begin{equation}\label{lipschitzestimateforrefractor}
%u(x_2) - u(x_1) \leq \phi(x_2, Y, X_2) - \phi(x_1, Y, X_2) \leq |\partial_x \phi(\tilde{x}, Y, X_2)||x_2 - x_1| \leq \frac{1}{\sqrt{\kappa^2 - 1}}|x_2 - x_1|.
u(x_2) - u(x_1) \leq \phi(x_2, Y, X_2) - \phi(x_1, Y, X_2)\leq \frac{1}{\sqrt{\kappa^2 - 1}}|x_2 - x_1|\quad \mbox{for all }x_1,x_2\in\Omega.
\end{equation}
If we interchange the roles of $x_1$ and $x_2$, we get a uniform Lipschitz bound for the refractors.

%\begin{Lemma} The smooth function $\phi(x,Y,X_0)$, defined in $\Omega\times\Sigma\times \mathcal{C}_{\Omega}$, has uniform derivatives bounds of any order in any entry, i.e.
%$$\left|D^{\alpha}\phi(x,Y,X_0)\right|\leq C_\alpha$$
%for any multi-index $\alpha=(\alpha_1,\ldots,\alpha_{3n+2})$ with a constant %$C_\alpha$ depending only on $\kappa,\Delta,\omega$.
%\end{Lemma}
%\begin{proof}

The above argument suggests that obtaining higher derivative estimates for $\phi$ will allow us to obtain higher derivative estimates for $u$. We calculate below the relevant derivatives of $\phi$ that will be used. 
Fix $(x,Y,X_0)\in\Omega\times\Sigma\times \mathcal{C}_{\Omega}$, and put $x_{n+1}=\phi(x)$. 

For the derivative in $x^0_{n+1}$, we notice that %\noindent Since $c(X,Y) = \kappa(y_{n+1} - x_{n+1}) - |X - Y|$, we have 
$\dfrac{\partial}{\partial x_{n+1}}c(X,Y) = -\kappa - \dfrac{(x_{n+1} - y_{n+1})}{|X - Y|}$ and thus $\left|\dfrac{\partial c(X,Y)}{\partial x_{n+1}}\right| \leq \kappa + 1$. 
From \eqref{eq:definitionofphi}
\begin{align*}
\frac{\partial \phi}{\partial x^0_{n+1}}(x, Y, X_0) & =  - \left(\frac{\kappa}{\kappa^2 - 1}\right) \frac{\partial c(X_0,Y)}{\partial x^0_{n+1}}- \left(\frac{c(X_0,Y)^2}{(\kappa^2 - 1)^2} + \frac{|x-y|^2}{\kappa^2 - 1} \right)^{-\frac{1}{2}} \frac{c(X_0,Y)}{(\kappa^2 - 1)^2} \frac{\partial  c(X_0,Y)}{\partial x^0_{n+1}} \\
& =  - \left[\left(\frac{\kappa}{\kappa^2 - 1}\right) + \left(\frac{c(X_0,Y)^2}{(\kappa^2 - 1)^2} + \frac{|x-y|^2}{\kappa^2 - 1} \right)^{-\frac{1}{2}} \frac{c(X_0,Y)}{(\kappa^2 - 1)^2} \right] \frac{\partial c(X_0,Y)}{\partial x^0_{n+1}} ,
\end{align*}
and so we get
\begin{align}
\left|\frac{\partial \phi}{\partial x^0_{n+1}}(x, Y, X_0) \right| & \leq \left[\left(\frac{\kappa}{\kappa^2 - 1}\right) + \left(\frac{c(X_0,Y)^2}{(\kappa^2 - 1)^2} + \frac{|x-y|^2}{\kappa^2 - 1} \right)^{-\frac{1}{2}} \frac{c(X_0,Y)}{(\kappa^2 - 1)^2} \right] \bigg|\frac{\partial c(X_0,Y)}{\partial x^0_{n+1}} \bigg| \nonumber \\
& \leq \left[\left(\frac{\kappa}{\kappa^2 - 1}\right) + \frac{1}{\kappa^2 - 1} \right] (\kappa + 1)= \frac{\kappa + 1}{\kappa - 1}. \label{derivlastcoord}
\end{align}

Next we calculate the second derivatives and get, for $i,j = 1, \ldots, n$, that
\begin{eqnarray}\label{2derij}
&&\,\,\,\frac{\partial^2 \phi}{\partial x_i \partial x_j}(x, Y, X_0)=\\
&=&-\frac{1}{\kappa^2 - 1} \left(\frac{c(X_0,Y)^2}{(\kappa^2 - 1)^2} + \frac{|x-y|^2}{\kappa^2 - 1} \right)^{-\frac{3}{2}} \left\{\delta_{ij} \left(\frac{c(X_0,Y)^2}{(\kappa^2 - 1)^2} + \frac{|x-y|^2}{\kappa^2 - 1} \right) - \frac{(x_i - y_i)(x_j - y_j)}{\kappa^2 - 1} \right\}.\nonumber
\end{eqnarray}
This gives
\begin{equation}\label{hessian}
\left|\frac{\partial^2 \phi}{\partial x_i \partial x_j}(x, Y, X_0) \right| \leq \frac{2}{\sqrt{c(X_0,Y)^2 + (\kappa^2 - 1)|x-y|^2}} \leq \frac{2}{c(X_0,Y)}. 
\end{equation}
The mixed second derivatives in $x$ and $Y$ are, for $i = 1, \ldots n$ and $j = 1, \ldots, n+1$,
\begin{eqnarray*}
\frac{\partial^2 \phi}{\partial x_i \partial y_j}(x, Y, X_0) &=& \frac{\partial}{\partial y_j} \left(\frac{y_i - x_i}{\sqrt{c(X_0,Y)^2 + (\kappa^2 - 1)|x-y|^2}} \right)=\\
&=&\frac{\delta_{ij}}{\left(c(X_0,Y)^2 + (\kappa^2 - 1)|x-y|^2 \right)^{\frac{1}{2}}}+\\
&&-\,\frac{(y_i - x_i)\left(c(X_0,Y)\frac{\partial}{\partial y_j}c(X_0,Y) + (\kappa^2 - 1)(y_j - x_j)(1 - \delta_{j, n+1}) \right)}{\left(c(X_0,Y)^2 + (\kappa^2 - 1)|x-y|^2 \right)^{\frac{3}{2}}}.
\end{eqnarray*}
Since $\frac{\partial}{\partial y_j}c(X,Y) = \kappa \delta_{j, n+1} - \frac{(y_j - x_j)}{|X - Y|}$, we have $|\frac{\partial}{\partial y_j}c(X,Y)| \leq \kappa + 1$. Therefore,
\begin{align}
\left|\frac{\partial^2 \phi}{\partial x_i \partial y_j}(x, Y, X_0) \right| & \leq \frac{c(X_0,Y)^2 + (\kappa + 1)|x-y| c(X_0,Y) + 2(\kappa^2 - 1)|x-y|^2}{\left(c(X_0,Y)^2 + (\kappa^2 - 1)|x-y|^2 \right)^{\frac{3}{2}}} \nonumber \\
%& \leq \left(2 + \frac{1}{2}\sqrt{\frac{\kappa + 1}{\kappa - 1}} \right) \left(\frac{1}{\left[c(X_0,Y)^2 + (\kappa^2 - 1)|x-y|^2 \right]^{\frac{1}{2}}} \right) \nonumber \\
& \leq \left(2 + \frac{1}{2}\sqrt{\frac{\kappa + 1}{\kappa - 1}} \right)\frac{1}{c(X_0,Y)}. \label{mixedhessian}
\end{align}

It is evident from the above calculations that in order to bound the derivatives of $\phi$ in a uniform manner, we must obtain a positive lower bound for $c(X,Y)$ when $X\in C_\Omega$ and $Y\in \Sigma$. For this, we will use the structural assumption \eqref{compatibilitycondition}.
In fact, let $X = (x, x_{n+1}) \in \mathcal{C}_{\Omega}$ and $Y = (y, y_{n+1}) \in \Sigma$. 
%From \eqref{eq:definitionofc(X,Y)} we get the upper bound
%$$c(X,Y)\leq (\kappa-1)\,|X-Y|\leq (\kappa-1)\sqrt{\Delta^2+\tau_2^2}.$$
%THIS SHOULD HAVE $\kappa+1$ BUT THE BOUND IS NOT USED.
%Moreover, our compatibility conditions give a positive lower bound for $c(X,Y)$. 
Since $|x - y| \leq \Delta$, we first have
$$c(X,Y) = \kappa(y_{n+1} - x_{n+1}) - |X - Y| \geq \kappa(y_{n+1} - x_{n+1}) - \sqrt{\Delta^2 + (y_{n+1} - x_{n+1})^2}.$$
Next, if we let $\psi(t) = \kappa t - \sqrt{\Delta^2 + t^2}$, then $\psi(0) = - \Delta < 0$ and $\psi'(t) = \kappa - \frac{t}{\sqrt{\Delta^2 + t^2}} \geq \kappa - 1 > 0$. Thus, $\psi(t) \geq - \Delta + (\kappa - 1)t$ for all $t \geq 0$. It then follows that if $\gamma > 0$ is given, we have
\begin{equation}\label{lowerestimateforC}
\inf\limits_{X \in \mathcal{C}_{\Omega}, Y \in \Sigma} (y_{n+1} - x_{n+1}) \geq \frac{\gamma + \Delta}{\kappa - 1} \  \Rightarrow  \ \inf\limits_{X \in \mathcal{C}_{\Omega}, Y \in \Sigma} c(X,Y) \geq \gamma.
\end{equation}
From \eqref{compatibilitycondition}, we get
\begin{equation}
\tau_1 - \tau_0 \geq \frac{\kappa \Delta}{\kappa - 1}  = \frac{(\kappa - 1)\Delta + \Delta}{\kappa - 1}
\end{equation}
Since $\inf\limits_{X \in \mathcal{C}_{\Omega}, Y \in \Sigma} (y_{n+1} - x_{n+1}) \geq \tau_1 - \tau_0$, we then obtain that $\gamma = (\kappa - 1) \Delta$ is a lower bound for $c(X,Y)$. 
%This provides us with uniform estimates for the derivatives of $\phi(x, Y, X)$.
%Our uniform lower bound $\inf\limits_{X \in \mathcal{C}_{\Omega}, Y \in \Sigma} c(X,Y) \geq (\kappa-1)\Delta$ from \eqref{lowerestimateforC} now allows us to bound %uniformly the second derivatives in 
Clearly, this bound yields uniform bounds in 
\eqref{hessian} and \eqref{mixedhessian}, as well as for higher order derivatives.
%\end{proof}
%The gradient of $\phi$ can be calculated to obtain
%\begin{equation}\label{hyperderiv}
%D\phi(x) = \frac{x-y}{y_{n+1} - x_{n+1} - \kappa|X-Y|}
%\end{equation}
%Note that $y_{n+1} - x_{n+1} \leq |X - Y| < \kappa|X - Y|$ as $\kappa > 1$, and so the denominator is always strictly negative. Thus, $D\phi(x) \cdot (x - y) < 0$.
%If $X=(x,\phi(x))$, $x_{n+1}=\phi(x)$, then 
%\begin{equation}
%D_x\phi(x)=\dfrac{y-x}{\sqrt{b^2+(\kappa^2-1)|x-y|^{2}}}=\dfrac{y-x}{\kappa|X-Y|-(y_{n+1}-x_{n+1})}.
%\end{equation}

We explicitly remark that the first order derivative bound in \eqref{derivative} is independent of the bounds for $c$, and thus independent of the compatibility assumptions. It depends just on the fact that the relevant supporting objects in our problem are hyperboloids, and it gives automatically global Lipschitz bounds for the refractor. This is in strong contrast with the case $\kappa<1$ considered in \cite{gutierrez-tournier:regularityparallelrefractor}. In fact, the supporting objects in \cite{gutierrez-tournier:regularityparallelrefractor} are ellipsoids and to obtain global Lipschitz bounds for them a condition between $\Omega$ and $\Sigma$ is needed, see \cite[Section 2.3]{gutierrez-tournier:regularityparallelrefractor}.
%\marginpar{reworded}
%$$\left|D_x\phi(x,Y,X_0)\right|=\dfrac{|y-x|}{\sqrt{c(X_0,Y)^{2}+(\kappa^2-1)|x-y|^{2}}}\leq 
%\dfrac{1}{\sqrt{\kappa^2-1}}.$$

%{\bf Remark:} The property of being supported by hyperboloids leads to a uniform Lipschitz bound for refractors. Indeed, by the gradient estimate \eqref{derivative} , we have for $x_1, x_2 \in \Omega$ and $X_2 = (x_2, u(x_2))$ that
%\begin{equation}\label{lipschitzestimateforrefractor}
%u(x_2) - u(x_1) \leq \phi(x_2, Y, X_2) - \phi(x_1, Y, X_2) \leq |\partial_x \phi(\tilde{x}, Y, X_2)||x_2 - x_1| \leq \frac{1}{\sqrt{\kappa^2 - 1}}|x_2 - x_1|.
%\end{equation}
%Interchanging the roles of $x_1$ and $x_2$ gives the desired uniform Lipschitz bound.

The derivative bounds and the properties of hyperboloids also imply the following estimates, which will be used in Section \ref{main}.

\begin{Lemma}\label{unposu} Let $\bar{X} \in \mathcal{C}_{\Omega}$, $Y \in \Sigma$, $x_0 \in \Omega$. Let also $X_0 = (x_0, \phi(x_0, Y, \bar{X}))$, and assume $X_0^* = (x_0, \phi(x_0, Y, \bar{X}) + h) \in \mathcal{C}_{\Omega}$ for some positive $h$. Then 
$$0 \leq \phi(x, Y, X_0^*) - \phi(x, Y, X_0) \leq \frac{\kappa + 1}{\kappa - 1}\, h \qquad\mbox{for all }x \in \Omega.$$
\end{Lemma}
\begin{proof} We recall that $c(X^0, Y) = \kappa(y_{n+1} - x^0_{n+1}) - |X^0 - Y|$. Since $\kappa > 1$ we have $\partial_{x^0_{n+1}}c(X^0, Y) = -\left(\kappa - \frac{y_{n+1} - x^0_{n+1}}{|Y - X_0|} \right) < 0$. Therefore,
$$\partial_{x^0_{n+1}} \phi(x, Y, X^0) = \left[-\frac{\kappa}{\kappa^2 - 1} - \left(\frac{c(X^0,Y)^2}{(\kappa^2 - 1)^2} + \frac{|x-y|^2}{\kappa^2 - 1} \right)^{-\frac{1}{2}} \frac{c(X^0,Y)}{(\kappa^2 - 1)^2} \right] \partial_{x^0_{n+1}}c(X^0, Y) > 0.$$
It follows that $\phi(x, Y, X_0^*) - \phi(x, Y, X_0) \geq 0$.
%, which provides us with the lower bound. 
For the upper bound we have, for some $\tilde{X} \in [X_0, X_0^*]$, that
$$\phi(x, Y, X_0^*) - \phi(x, Y, X_0) = \partial_{x^0_{n+1}} \phi(x, Y, \tilde{X}) h \leq \frac{\kappa + 1}{\kappa - 1} h,$$
where in the last inequality we have used \eqref{derivlastcoord}. 
\end{proof}

The following lemma can be proved verbatim as in \cite[Lemma 2.3]{gutierrez-tournier:regularityparallelrefractor}.

\begin{Lemma}\label{lemmafour} There exists $C > 0$ such that for all $Y, \bar{Y} \in \Sigma$ and $X_0 \in \mathcal{C}_{\Omega}$, we have
$$|\phi(x,Y,X_0) - \phi(x,\bar{Y},X_0)| \leq C|x-x_0||Y-\bar{Y}|.$$
\end{Lemma}

\setcounter{equation}{0}
\section{Regularity assumptions on the target}\label{asstar}

We assume the following assumptions on the target $\Sigma$.

\subsection{Parametrization of the target}

%The following fact, which is a simple consequence of Snell's law, will be used on multiple occasions throughout this paper:
%\begin{center}
%\textcolor{red}{if the upper focus of the hyperboloid defined by $\phi(x, Y, X_0)$ satisfies $Y = X_0 + s \Lambda(v)$, for some $s > 0$ and $v \in \mathbb{R}^n$, then $v = D_x\phi(x_0, Y, X_0)$}.
%\end{center}

Let us assume that each $Y \in \Sigma$ can be represented in the form $Y = X + s_X(\Lambda)\Lambda$ for each $X \in \mathcal{C}_{\Omega}$, with $|\Lambda| = 1$ and $s_X(\Lambda)$ Lipschitz as a function of $\Lambda$.\\
The Lipschitz character of $s_X$, together with \eqref{Fact} and the estimate \eqref{mixedhessian}, implies that for any $X_0 \in \mathcal{C}_{\Omega}$ there exists $C = C(X_0) \geq 1$ such that
%\begin{Lemma} Let $X_0 \in \mathcal{C}_{\Omega}$, $\bar{Y}, \hat{Y} \in \Sigma$, $\bar{v}, \hat{v} \in \mathbb{R}^n$, $\bar{s}, \hat{s} > 0$, where $\bar{Y} = X_0 + s_{X_0}(\Lambda(\bar{v}))\Lambda(\bar{v})$, $\hat{Y} = X_0 + s_{X_0}(\Lambda(\hat{v}))\Lambda(\hat{v})$. Then there exists $C = C(X_0) \geq 1$ such that
\begin{equation}\label{LipschitzEstimate}
\frac{1}{C}|\bar{Y} - \hat{Y}| \leq |\bar{v} - \hat{v}| \leq C|\bar{Y} - \hat{Y}|
\end{equation}
for all $\bar{Y}, \hat{Y} \in \Sigma$, $\bar{v}, \hat{v} \in \mathbb{R}^n$, $\bar{s}, \hat{s} > 0$ satisfying $\bar{Y} = X_0 + \bar{s}\Lambda(\bar{v})$ and $\hat{Y} = X_0 + \hat{s}\Lambda(\hat{v})$, see \cite[Lemma 2.1]{gutierrez-tournier:regularityparallelrefractor}.

\subsection{Regularity of the target}

Given $\bar{Y}, \hat{Y} \in \Sigma$ and $X_0 \in \mathcal{C}_{\Omega}$, let $\bar{v} = D_x\phi(x_0, \bar{Y}, X_0)$, $\hat{v} = D_x\phi(x_0, \hat{Y}, X_0)$. Let us consider 
%$$v(\lambda) = (1 - \lambda)\bar{v} + \lambda \hat{v}, \qquad\lambda \in [0,1].$$ 
%We want to consider the set of points 
$$C(X_0,\bar{Y}, \hat{Y}) = \left\{X_0 + s \Lambda(v(\lambda)) : s > 0,\, v(\lambda) = (1 - \lambda)\bar{v} + \lambda \hat{v}\,\,\mbox{for } \lambda \in [0,1] \right\}.$$ 
From \eqref{Fact}, we know that if $Y(\lambda) = X_0 + s \Lambda(v(\lambda)) \in C(X_0, \bar{Y}, \hat{Y})$, then $v(\lambda)=D_x\phi(x_0, Y(\lambda), X_0)$. Define 
$$[\bar{Y}, \hat{Y}]_{X_0} := \Sigma \cap C(X_0, \bar{Y}, \hat{Y}).$$ 
By the parametrization of the target and \eqref{LipschitzEstimate}, each $X_0 + s \Lambda(v(\lambda)) \in C(X_0, \bar{Y}, \hat{Y})$ intersects $\Sigma$ in at most one point for each $\lambda \in [0,1]$. The points in $[\bar{Y}, \hat{Y}]_{X_0}$ have the form $Y(\lambda) = X_0 + s_{X_0}(\Lambda(v(\lambda)))\Lambda(v(\lambda))$ for $\lambda \in [0,1]$.

\begin{Definition}\label{regularityhypothesis} Fix $X_0 \in \mathcal{C}_{\Omega}$. We say the target $\Sigma$ is regular from $X_0$ if there exists a neighborhood $U_{X_0}$ and $C_1, C_2 > 0$ depending on $U_{X_0}$ such that, for all $\bar{Y}, \hat{Y} \in \Sigma$ and $Z = (z,z_{n+1}) \in U_{X_0}$, we have
\begin{equation}\label{maximumprincipleforrefractor}
\phi(x, Y_Z(\lambda), Z) \leq \max \left\{\phi(x,\bar{Y},Z), \phi(x,\hat{Y},Z) \right\} - C_1|\bar{Y} - \hat{Y}|^2|x-z|^2
\end{equation}
for all $x \in \Omega$ with $|x-z| \leq C_2$, $\lambda \in [\frac{1}{4}, \frac{3}{4}]$ and $Y_Z(\lambda) = Z + s_Z(\Lambda(v(\lambda))) \Lambda(v(\lambda))$.
\end{Definition}

The following characterization of the regularity from a point in $\mathcal{C}_{\Omega}$ can be proved exactly as in \cite[Theorem 3.2]{gutierrez-tournier:regularityparallelrefractor}.

\begin{Theorem} The target $\Sigma$ is regular from $X_0 \in \mathcal{C}_{\Omega}$ if and only if there exists a neighborhood $U_{X_0}$ and $C(X_0) > 0$ such that, for all $Y_0 \in \Sigma$, $Z \in U_{X_0}$ and for all $\xi,\eta\in\R^n$ with $\xi \perp \eta$, we have
\begin{equation}\label{MTW}
\frac{d^2}{d\epsilon^2} \bigg|_{\epsilon = 0} \left\langle D^2_x \phi(x_0, Y_{\epsilon}, X_0) \eta, \eta \right\rangle \geq C(X_0)|\xi|^2 |\eta|^2
\end{equation}
where $Y_{\epsilon} = Z + s_Z(\Lambda(v + \epsilon \xi)) \Lambda(v + \epsilon \xi)$ and $v = D_x\phi(z, Y_0, Z)$.
\end{Theorem}

The theorem above follows from \cite[Theorem 3.2]{gutierrez-tournier:regularityparallelrefractor} since the proof of that theorem does not rely on the particular structure of the function $\phi$ nor on the size of the refractive index $\kappa$. Indeed, the condition \eqref{maximumprincipleforrefractor} is satisfied by the negative of the function in \cite[Theorem 3.2]{gutierrez-tournier:regularityparallelrefractor}, and so the condition \eqref{MTW} has the opposite sign as well.

Let us end this section with some clarifying remarks. The set $[\bar{Y}, \hat{Y}]_{X_0}$ mimics the notion of a $c$-segment in the theory of optimal mass transport (cf. \cite{villani:stfleurBOOK}), while the condition \eqref{MTW} is akin to the Ma-Trudinger-Wang condition (A3) in the regularity theory of optimal transport maps (cf. \cite{MaTrudingerWang:regularityofpotentials}, \cite{villani:stfleurBOOK}). A watershed for the regularity theory of mass transport is the result of Loeper \cite{loeper:actapaper}, which shows the condition (A3) is equivalent to a maximum principle for $c$-support functions. This forms the motivation for the regularity hypothesis \eqref{maximumprincipleforrefractor} and the theorem above is the analog of this characterization for the case of the parallel refractor.

%WE NEED TO EXPLAIN HERE THAT THIS FOLLOWS FROM G-T MULTIPLYING BY -1.

\setcounter{equation}{0}
\section{Local to Global}\label{sec:localtoglobal}

%\textcolor{red}{Introduction to this Section for the (A3)-type condition. Or maybe at the end of the Section}. See \cite[Remark 3.3]{gutierrez-tournier:regularityparallelrefractor}.

%\subsection{"Double-Mountain above Sliding-Mountain"-result for hyperboloids}

%For a function $u : \Omega \rightarrow [0,\tau_0]$, $x_0 \in \Omega$, $X_0 = (x_0, u(x_0))$, the normal map is defined as 
%\begin{equation}
%F_u(x_0) = \left\{Y \in \Sigma : u(x) \geq \phi(x, Y, X_0) \ \forall x \in \Omega \right\}.
%\end{equation}
%By Definition \ref{def:refractor}, $u$ is a parallel refractor if $F_u(x_0) \neq \emptyset$ for all $x_0 \in \Omega$.
%%%%%%%%%%%%%%%%%%%%%%%%%%%%%%%%%%%%%%%%%%%%%%%%%%%%%%%%%%%%%%%%%%%%%%%%%%%%%%%%%%%%%%%%%%%%%%%%%%%%%%%%%%%%%%%%%%%%%%

%\noindent{\bf Condition (AW):} 

Loeper's maximum principle allowed him to obtain a result which Kim and McCann \cite{KimMcCann:DASMpaper} refer to as the {\emph{DASM}} (Double-Mountain Above Sliding-Mountain) Theorem in the context of optimal mass transport. This in turn enabled Loeper to obtain a local-implies-global result for $c$-support functions. This section is devoted to establishing the analog of this local-implies-global result in the setting of the parallel refractor. We refer the reader to the end of this section for further comments.%after the proof of Theorem \ref{localtoglobal}.

We say that the target $\Sigma$ satisfies condition (AW) from $X_0 \in \mathcal{C}_{\Omega}$ if for all $Y_0 \in \Sigma$ written as $Y_0 = X_0 + s_{X_0}(\Lambda(v)) \Lambda(v)$, and for all $\xi \perp \eta$, we have
%\begin{center}
%\displaystyle{
\begin{equation}\label{eq:AWcondition}
\frac{d^2}{d\epsilon^2} \bigg|_{\epsilon = 0} \left\langle D^2_x \phi(x_0, Y_{\epsilon}, X_0) \eta, \eta \right\rangle \geq 0,\tag{AW}
\end{equation}
%\end{center}
where $Y_{\epsilon} = X_0 + s_{X_0}(\Lambda(v + \epsilon \xi)) \Lambda(v + \epsilon \xi)$, and $v = D_x\phi(x_0, Y_0, X_0)$. 
Equivalently, if $Y(v) = X_0 + s_{X_0}(\Lambda(v)) \Lambda(v)$, then the condition (AW) requires that for all $\xi \perp \eta$, we have
%\begin{center}
%\displaystyle{
$$\sum_{i,j,k,l}D_{v_{\ell},v_k} \left[D_{x_i,x_j} \phi(x_0, Y(v), X_0) \eta_i \eta_j \right]\xi_k \xi_{\ell} \geq 0.$$
%\end{center}
%at $x = x_0$ and $X = X_0 = (x_0, x^0_{n+1})$.
\vskip 0.6cm

Let us put $H(v,X) = s_X(\Lambda(v)) Q(v)$, where we recall that $\Lambda(v) = (Q(v)v, -Q(v) + \kappa)$ as defined in \eqref{eq:defofLambdav}. %$Q(v) = \frac{-\delta}{\sqrt{1 + |v|^2}}$, and $\delta < 0$ since $\kappa > 1$ (indeed, $\delta = \frac{-\kappa + \sqrt{1 - (\kappa^2-1)|v|^2}}{\sqrt{1+|v|^2}}$).Then $\Lambda(v) = (Q(v)v, -Q(v) + \kappa)$. 
We denote $J(Y, X, \eta) = \left\langle D^2_x \phi(x, Y, X) \eta, \eta \right\rangle$. %Notice that it suffices to check the condition (AW) for $\eta$ satisfying $|\eta| = 1$. We recall that
%\begin{center}
%$\displaystyle{\frac{\partial^2 \phi}{\partial x_i \partial x_j}(x, Y, X) = -\frac{1}{\kappa^2 - 1} \left(\frac{c(X,Y)^2}{(\kappa^2 - 1)^2} + \frac{|x-y|^2}{\kappa^2 - 1} \right)^{-\frac{3}{2}} \left\{\delta_{ij} \left(\frac{c(X,Y)^2}{(\kappa^2 - 1)^2} + \frac{|x-y|^2}{\kappa^2 - 1} \right) - \frac{(x_i - y_i)(x_j - y_j)}{\kappa^2 - 1} \right\}}$.
%\end{center}
\noindent 
By \eqref{2derij}, for $|\eta| = 1$, we have
%\begin{align*}
%J(Y, X, \eta) & = \left\langle D^2_x \phi(x, Y, X) \eta, \eta \right\rangle \\
%& = \sum\limits_{i,j = 1}^n \frac{\partial^2 \phi}{\partial x_i \partial x_j}(x, Y, X)\eta_i \eta_j\\
%& = -\frac{1}{\kappa^2 - 1} \left(\frac{c(X,Y)^2}{(\kappa^2 - 1)^2} + \frac{|x-y|^2}{\kappa^2 - 1} \right)^{-\frac{3}{2}} \left\{\frac{c(X,Y)^2}{(\kappa^2 - 1)^2} + \frac{|x-y|^2}{\kappa^2 - 1}  - \frac{\left\langle x - y, \eta \right\rangle^2}{\kappa^2 - 1} \right\}
%\end{align*}
$$J(Y, X, \eta)=-\frac{1}{\kappa^2 - 1} \left(\frac{c(X,Y)^2}{(\kappa^2 - 1)^2} + \frac{|x-y|^2}{\kappa^2 - 1} \right)^{-\frac{3}{2}} \left\{\frac{c(X,Y)^2}{(\kappa^2 - 1)^2} + \frac{|x-y|^2}{\kappa^2 - 1}  - \frac{\left\langle x - y, \eta \right\rangle^2}{\kappa^2 - 1} \right\}.$$
Since $c(X,Y)^2 + (\kappa^2 - 1)|x - y|^2 = (\kappa|X - Y| - (y_{n+1} - x_{n+1}))^2$, we obtain
%\begin{center}
%$\displaystyle{
$$J(Y, X, \eta) = -(\kappa|X - Y| - (y_{n+1} - x_{n+1}))^{-1} + (\kappa^2 - 1)\left\langle x - y, \eta \right\rangle^2(\kappa|X - Y| - (y_{n+1} - x_{n+1}))^{-3}.$$
%\end{center}

\noindent If $Y - X = s_X(\Lambda(v))(Q(v)v, -Q(v) + \kappa)$, then $|Y - X| = s_X(\Lambda(v))$, and so
%\begin{center}
%$\displaystyle{
$$\kappa|X - Y| - (y_{n+1} - x_{n+1}) = \kappa s_X(\Lambda(v)) +  s_X(\Lambda(v)) Q(v) - \kappa s_X(\Lambda(v)) = s_X(\Lambda(v)) Q(v).$$
%\end{center}

\noindent Hence,
\begin{equation}
J(Y, X, \eta) = \frac{(\kappa^2 - 1)\left\langle v, \eta \right\rangle^2 - 1}{s_X(\Lambda(v))Q(v)}.
\end{equation}

\noindent Now if $v_{\epsilon} = v + \epsilon \xi$, then we can write $Y_{\epsilon} - X = s_X(\Lambda(v_{\epsilon}))(Q(v_{\epsilon})v_{\epsilon}, -Q(v_{\epsilon}) + \kappa)$. Recalling that $|Y_{\epsilon} - X| = s_X(\Lambda(v_{\epsilon}))$, we therefore obtain
%\begin{center}
%$\displaystyle{
$$J(Y_{\epsilon}, X, \eta) = \frac{(\kappa^2 - 1)\left\langle v_{\epsilon}, \eta \right\rangle^2 - 1}{s_X(\Lambda(v_{\epsilon}))Q(v_{\epsilon})} = \frac{(\kappa^2 - 1)\left\langle v_{\epsilon}, \eta \right\rangle^2 - 1}{H(v_{\epsilon},X)} =: F(v_{\epsilon},X,\eta).$$
%\end{center}

\noindent The condition (AW) is then equivalent to
\begin{center}
$\displaystyle{\frac{d^2}{d\epsilon^2} \bigg|_{\epsilon = 0} F(v_{\epsilon}, X, \eta) \geq 0}$.
\end{center}

\noindent On the other hand, by setting $G(v,X) = H(v,X)^{-1}$, we see that 
\begin{center} 
$\displaystyle{F(v, X, \eta) = \left((\kappa^2 - 1)\left\langle v, \eta \right\rangle^2 - 1 \right) G(v,X)}$.
\end{center}

\noindent Notice that if $\xi \perp \eta$, then

$$\frac{d}{d \epsilon} \left[(\kappa^2 - 1)\left\langle v_{\epsilon}, \eta \right\rangle^2 - 1 \right] = 2 (\kappa^2 - 1) \left\langle v_{\epsilon}, \eta \right\rangle \left\langle \xi, \eta \right\rangle  = 0.$$

\noindent It follows that
\begin{equation}
\frac{d^2}{d\epsilon^2} \bigg|_{\epsilon = 0} F(v_{\epsilon}, X, \eta) = \left[(\kappa^2 - 1) \left\langle v, \eta \right\rangle^2 - 1 \right] \left\langle D^2_v G (v, X) \xi, \xi \right\rangle.
\end{equation}

\noindent By the derivative estimate \eqref{derivative} for $\phi(x,Y,X)$, we have $|v| \leq \frac{1}{\sqrt{\kappa^ 2 - 1}}$ (actually it is strictly less than 1). Thus, $(\kappa^2 - 1) \left\langle v, \eta \right\rangle^2 - 1 < 0$ for all $v$ and $|\eta| = 1$. Therefore, the condition \eqref{eq:AWcondition} implies $G(\cdot, X)$ is a positive concave function for each $X$.

%%%%%%%%%%%%%%%%%%%%%%%%%%%%%%%%%%%%%%%%%%%%%%%%%%%%%%%%%%%%%%%%%%%%%%%%%%%%%%%%%%%%%%%%%%%%%%%%%%%%%%%%%%%%%%%%%%%%%%

\begin{Theorem}\label{maxprinciple} Suppose that the condition \eqref{eq:AWcondition} holds from some $X_0 \in \mathcal{C}_{\Omega}$. Let $\bar{Y}, \hat{Y} \in \Sigma$ be given by $\bar{Y} = X_0+s_{X_0}(\Lambda(\bar{v}))\Lambda(\bar{v})$ and $\hat{Y} = X_0+s_{X_0}(\Lambda(\hat{v}))\Lambda(\hat{v})$. For $\lambda \in (0,1)$, let $v_{\lambda} = (1 - \lambda)\bar{v} + \lambda \hat{v}$, and define $Y_{\lambda} = X_0+ s_{X_0}(\Lambda(v_{\lambda}))\Lambda(v_{\lambda})$. Denoting $\mathcal{H}(Y) = \left\{X \in \mathbb{R}^{n+1} : c(X,Y) \geq c(X_0,Y) \right\}$, %, where as before, $c(X,Y) = \kappa(y_{n+1} - x_{n+1}) - |X - Y|$ and $\kappa > 1$.
then $$\H(Y_{\lambda}) \subseteq \mathcal{H}(\bar{Y}) \cup \mathcal{H}(\hat{Y}).$$ In particular, for all $x\in\Omega$, $$\phi(x, Y_{\lambda}, X_0) \leq \max \left\{\phi(x,\bar{Y}, X_0), \phi(x,\hat{Y}, X_0) \right\}.$$
\end{Theorem}

\begin{proof} Assume for simplicity $X_0=0$. Let us first make note of the following:
\begin{align}\label{simplelemma}
c(X,Z) \geq c(0,Z) & \Leftrightarrow \kappa(z_{n+1} - x_{n+1}) - |X - Z| \geq \kappa z_{n+1} - |Z| \\
& \Leftrightarrow |X - Z| \leq |Z| - \kappa x_{n+1} \notag\\
& \Rightarrow |X|^2 - 2 X \cdot Z + |Z|^2 \leq |Z|^2 - 2 \kappa x_{n+1} |Z| + \kappa^2 x_{n+1}^2 \notag\\
& \Rightarrow |X|^2 - \kappa^2 x_{n+1}^2 \leq 2 X \cdot Z - 2\kappa x_{n+1} |Z|.\notag \end{align}

We notice that the set $\bd\mathcal{H}(\hat{Y}) \cap \bd\mathcal{H}(Y_{\lambda})$ is contained in a hyperplane $\hat{T}$. Indeed, if $X \in \bd\mathcal{H}(\hat{Y}) \cap \bd\mathcal{H}(Y_{\lambda})$, then $c(X,\hat{Y}) = c(0, \hat{Y})$ and $c(X,Y_{\lambda}) = c(0, Y_{\lambda})$. Applying \eqref{simplelemma} in the case of equality with $Z = \hat{Y}$ and $Z = Y_{\lambda}$, we obtain

\begin{center}
$|X|^2 - \kappa^2 x_{n+1}^2 = 2 X \cdot \hat{Y} - 2\kappa x_{n+1} |\hat{Y}|$;\\
$|X|^2 - \kappa^2 x_{n+1}^2 = 2 X \cdot Y_{\lambda} - 2\kappa x_{n+1} |Y_{\lambda}|$.
\end{center}

\noindent Hence,

\begin{center}
$\displaystyle{X \cdot \hat{Y} - \kappa x_{n+1} |\hat{Y}| = X \cdot Y_{\lambda} - \kappa x_{n+1} |Y_{\lambda}|}$.
\end{center}

\noindent We can rewrite this as $X \cdot \hat{\eta} = 0$, where

\begin{center}
$\displaystyle{\hat{\eta} = (\hat{Y} - \kappa|\hat{Y}|e_{n+1}) - (Y_{\lambda} - \kappa|Y_{\lambda}|e_{n+1})}$.
\end{center}

\noindent In conclusion, $X \in \bd\mathcal{H}(\hat{Y}) \cap \bd\mathcal{H}(Y_{\lambda})$ implies $X \cdot \hat{\eta} = 0$. Analogously, if $X \in \bd\mathcal{H}(\bar{Y}) \cap \bd\mathcal{H}(Y_{\lambda})$, then $X \cdot \bar{\eta} = 0$ where 

\begin{center}
$\displaystyle{\bar{\eta} = (\bar{Y} - \kappa|\bar{Y}|e_{n+1}) - (Y_{\lambda} - \kappa|Y_{\lambda}|e_{n+1})}$.
\end{center}

\noindent Recall that any $Y \in \Sigma$ is given by $Y = s_0(\Lambda(v)) \Lambda(v)$. 
%where $\Lambda(v) = (Q(v)v, -Q(v) + \kappa)$ and $Q(v) = \frac{-\delta}{\sqrt{1 + |v|^2}}$ for $\delta < 0$. 
Denoting $H(v) := H(v,0)=s_0(\Lambda(v)) Q(v)$, 
%where $H(v,X) = s_X(\Lambda(v)) Q(v)$. Since $H(v) = s_0(\Lambda(v)) Q(v)$, w
we may rewrite $\hat{\eta}$ and $\bar{\eta}$ as the $(n+1)$-vectors

\begin{center}
$\displaystyle{\hat{\eta} = (H(\hat{v})\hat{v} - H(v_{\lambda}) v_{\lambda}, H(v_{\lambda})- H(\hat{v})), \ \bar{\eta} = (H(\bar{v})\bar{v} - H(v_{\lambda}) v_{\lambda}, H(v_{\lambda})- H(\bar{v}))}$.
\end{center}

\noindent This is because the first $n$ components of $Y$ are given by the vector $H(v)v$, while the $(n+1)$-st component is given by $Y_{n+1} = |Y|(-Q(v) + \kappa)$, since $|\Lambda(v)| = 1$. Hence, $Y_{n+1} - \kappa|Y| = -|Y| Q(v)$. Since $H(v) = s_0(\Lambda(v)) Q(v) = |Y| Q(v)$, it follows that $Y_{n+1} - \kappa|Y| = -H(v)$. Let us now prove a couple of claims.

\begin{enumerate}

\item[(1)] Suppose $X \in \mathcal{H}(Y_{\lambda})$. Then
$$X \cdot \hat{\eta} \geq 0 \quad\Rightarrow\quad X \in \mathcal{H}(\hat{Y}),$$
$$X \cdot \bar{\eta} \geq 0\quad\Rightarrow\quad X \in \mathcal{H}(\bar{Y}).$$
As a matter of fact, since $X \in \mathcal{H}(Y_{\lambda})$, we have $c(X, Y_{\lambda}) \geq c(0, Y_{\lambda})$, which implies by \eqref{simplelemma}

\begin{center}
$\displaystyle{|X|^2 - 2 X \cdot Y_{\lambda} + 2 \kappa x_{n+1} |Y_{\lambda}| - \kappa^2 x_{n+1}^2 \leq 0}$.
\end{center}
If $X \cdot \hat{\eta} \geq 0$, we then have 
\begin{center}
$\displaystyle{X \cdot (Y_{\lambda} - \hat{Y}) - \kappa x_{n+1}(|Y_{\lambda}| - |\hat{Y}|) \leq 0}$.
\end{center}
The last two inequalities give
\begin{center}
$\displaystyle{|X|^2 - 2 X \cdot \hat{Y} + 2 \kappa x_{n+1} |\hat{Y}| - \kappa^2 x_{n+1}^2 \leq 0}$.
\end{center}
In particular, this implies $|X - \hat{Y}|^2 \leq (|\hat{Y}| - \kappa x_{n+1})^2$. We want to conclude that $|X - \hat{Y}| \leq |\hat{Y}| - \kappa x_{n+1}$. This is possible thanks to the structural assumption \eqref{compatibilitycondition}. In fact, for any $Y \in \Sigma$ and $X = (x, x_{n+1}) \in \mathcal{C}_{\Omega}$, we have
\begin{center}
$\displaystyle{|Y| - \kappa x_{n+1} \geq y_{n+1} - \kappa x_{n+1} \geq \tau_1 - \kappa \tau_0 \geq 0}$.
\end{center}
This gives $|X - \hat{Y}| \leq |\hat{Y}| - \kappa x_{n+1}$, which implies $c(X,\hat{Y}) \geq c(0,\hat{Y})$. Thus, assuming $X \cdot \hat{\eta} \geq 0$, we have $X \in \mathcal{H}(\hat{Y})$, and the first implication is proved. The proof of the second implication we claimed is completely analogous.
%
%\bigskip
%
%\item If $X \cdot \bar{\eta} \geq 0$ and $X \in \mathcal{H}(Y_{\lambda})$, then $X \in \mathcal{H}(\bar{Y})$.\\
%
%The proof is entirely similar to that of the first claim.
\end{enumerate}
Let us note that so far, we have not used the condition \eqref{eq:AWcondition}, which we have shown to be equivalent to the concavity of $\frac{1}{H(v)}$. We will now use this fact in the proof of the following claim.

\begin{enumerate}

\item[(2)] If $X \cdot \hat{\eta} < 0$ and $X \cdot \bar{\eta} < 0$, then $X \notin \mathcal{H}(Y_{\lambda})$.
\bigskip

\noindent Assume $X \cdot \hat{\eta} < 0$ and $X \cdot \bar{\eta} < 0$. Notice that $$\mathcal{H}(Y_{\lambda}) \backslash \left\{0 \right\} \subset \left\{X \in \mathbb{R}^{n+1} : X \cdot N_{\lambda} > 0 \right\},$$ where $N_{\lambda} = (v_{\lambda}, -1)$. Hence, it is enough to show that $X \cdot N_{\lambda} \leq 0$. First assume $H(v_{\lambda}) - H(\bar{v}) \neq 0$ and $H(v_{\lambda}) - H(\hat{v}) \neq 0$. We will show that
%\begin{center}
%$\displaystyle{
$$N_{\lambda} = \left(\frac{1 - t}{H(\bar{v}) - H(v_{\lambda})} \right) \bar{\eta} + \left(\frac{t}{H(\hat{v}) - H(v_{\lambda})} \right) \hat{\eta}.$$
%\end{center}
with $\frac{1 - t}{H(\bar{v}) - H(v_{\lambda})} \geq 0$ and $\frac{t}{H(\hat{v}) - H(v_{\lambda})} \geq 0$ for some $t$. By comparing the first $n$ components of $\bar{\eta}$ and $\hat{\eta}$, we find that the above equality holds if and only if
%\begin{center}
%$\displaystyle{
\begin{eqnarray*}
&&v_{\lambda} = \frac{(1 - t)(H(\bar{v})\bar{v} - H(v_{\lambda})v_{\lambda})}{H(\bar{v}) - H(v_{\lambda})} + \frac{t(H(\hat{v})\hat{v} - H(v_{\lambda})v_{\lambda})}{H(\hat{v}) - H(v_{\lambda})}\\
%\end{center}
%\begin{center}
%$\displaystyle{
&& \hphantom{v_{\lambda}}=(1 - t) \left(v_{\lambda} + \frac{H(\bar{v})(\bar{v} - v_{\lambda})}{H(\bar{v}) - H(v_{\lambda})}\right) + t \left(v_{\lambda} + \frac{H(\hat{v})(\hat{v} - v_{\lambda})}{H(\hat{v}) - H(v_{\lambda})}\right)\\
%\end{center}
%\begin{center}
%$\displaystyle{
&\Leftrightarrow&(\hat{v} - \bar{v})\left\{\frac{(1-t)\lambda H(\bar{v})}{H(v_{\lambda}) - H(\bar{v})} - \frac{t(1 - \lambda) H(\hat{v})}{H(v_{\lambda}) - H(\hat{v})} \right\} = 0.
\end{eqnarray*}
%\end{center}
Therefore, we choose $t$ such that
%\begin{center}
%$\displaystyle{
$$\frac{(1-t)\lambda H(\bar{v})}{H(v_{\lambda}) - H(\bar{v})} = \frac{t(1 - \lambda) H(\hat{v})}{H(v_{\lambda}) - H(\hat{v})}.$$
%\end{center}
Since $Q(v) > 0$, we have $\lambda H(\bar{v}) > 0$ and $(1 - \lambda)H(\hat{v}) > 0$. It follows that $\frac{(1-t)}{H(v_{\lambda}) - H(\bar{v})}$ and  $\frac{t}{H(v_{\lambda}) - H(\hat{v})}$ have the same sign. From the last identity, we also obtain
%
%\begin{center}
%$\displaystyle{
$$\lambda H(\bar{v}) = \frac{t}{H(v_{\lambda}) - H(\hat{v})} \left\{H(v_{\lambda})((1 - \lambda)H(\hat{v}) + \lambda H(\bar{v})) - H(\bar{v})H(\hat{v}) \right\}.$$
%\end{center}
By the concavity of $1/H$, we have $H(v_{\lambda})((1 - \lambda)H(\hat{v}) + \lambda H(\bar{v})) - H(\bar{v})H(\hat{v}) \leq 0$. Hence, $\frac{t}{H(v_{\lambda}) - H(\hat{v})} \leq 0$ and thus $\frac{(1-t)}{H(v_{\lambda}) - H(\bar{v})} \leq 0$ as well.

Now consider the case $H(v_{\lambda}) - H(\hat{v}) = 0$ and $H(v_{\lambda}) - H(\bar{v}) \neq 0$. From the concavity of $1/H$, we have $H(v_{\lambda}) < H(\bar{v})$. If we write $N_{\lambda} = \frac{1}{H(v_{\lambda}) - H(\bar{v})} \bar{\eta} + t \hat{\eta}$, then $t = \frac{\lambda H(\bar{v})}{(H(v_{\lambda}) - H(\bar{v}))(1-\lambda)H(\hat{v})}$, and so $t < 0$. 

Finally, the last case to consider is $H(v_{\lambda}) = H(\bar{v}) = H(\hat{v})$. In such case $\bar{\eta} = \lambda H(\bar{v})(\hat{v} - \bar{v}, 0)$ and $\hat{\eta} = (1-\lambda)H(\bar{v})(\bar{v} - \hat{v}, 0)$, and so both inequalities $X \cdot \hat{\eta} <0$ and $X \cdot \bar{\eta} < 0$ cannot hold simultaneously. 
\end{enumerate}
\bigskip

The above claims complete the proof of the theorem. Indeed, by the second claim, if $X \in \H(Y_{\lambda})$, then either $X \cdot \hat{\eta} \geq 0$ or $X \cdot \bar{\eta} \geq 0$. Hence, it follows from the first claim that $X \in \H(\hat{Y})$ or $X \in \H(\bar{Y})$.
\end{proof}

For a function $u : \Omega \rightarrow [0,\tau_0]$, $x_0 \in \Omega$, $X_0 = (x_0, u(x_0))$, the refractor-normal map is defined as 
\begin{equation}
F_u(x_0) = \left\{Y \in \Sigma : u(x) \geq \phi(x, Y, X_0) \ \forall x \in \Omega \right\}.
\end{equation}
By Definition \ref{def:refractor}, $u$ is a parallel refractor if $F_u(x_0) \neq \emptyset$ for all $x_0 \in \Omega$.

The next result applies the previous theorem to show that a locally supporting hyperboloid is in fact a globally supporting one under the condition \eqref{eq:AWcondition}. The proof is similar to that of  \cite[Proposition 4.2]{gutierrez-tournier:regularityparallelrefractor}; we provide all the details for the convenience of the reader.

\begin{Theorem}\label{localtoglobal} Suppose $u$ is a parallel refractor in $\Omega$ that satisfies the condition \eqref{eq:AWcondition} from $X_0 = (x_0, u(x_0))$, with $x_0 \in \Omega$. If there exists $Y_0 \in \Sigma$ and $\epsilon > 0$ such that $u(x) \geq \phi(x,Y_0,X_0)$ for all $x \in B_{\epsilon}(x_0)$, then $u(x) \geq \phi(x,Y_0,X_0)$ for all $x \in \Omega$. \end{Theorem}

\begin{proof} Denote the refractor by $\mathcal{R}=\{(x,u(x))\,:\, x\in\Omega\}$, and let $Y(X, v) := X + s_{X}(\Lambda(v))\Lambda(v)$ where $X \in \mathcal{R}$. Consider the local sub differential
$$\partial u (x_0) = \left\{v \in \mathbb{R}^n : u(x) \geq u(x_0) + v \cdot (x - x_0) + o(|x-x_0|) \right\}.$$
Notice that if $u(x) \geq \phi(x,Y_0,X_0)$ for all $x \in B_{\epsilon}(x_0)$, then $Y_0 \in \left\{Y(X_0,v) : v \in \partial u (x_0) \right\}$. Indeed, by using the Taylor expansion of $\phi$ around $x_0$, we obtain
\begin{align}\label{TayExp}
u(x) \geq \phi(x, Y_0, X_0) & = \phi(x_0, Y_0, X_0) + D_x\phi(x_0,Y,X_0)\cdot (x - x_0) + o(|x - x_0|) \nonumber \\
& = u(x_0) + D_x\phi(x_0, Y, X_0) \cdot (x - x_0) + o(|x - x_0|).
\end{align}
It follows that $v_0 := D_x\phi(x_0,Y,X_0) \in \partial u (x_0)$. By \eqref{Fact}, we recall that $v = D_x \phi(x, Y, X)$ with $Y \in \Sigma$ if and only if $Y = Y(X,v)$. Therefore, $Y_0 = Y(X_0, v_0)$ with $v_0 \in \partial u (x_0)$ as claimed.

We will now show that under the condition \eqref{eq:AWcondition}, we have the inclusion
\begin{equation}
\left\{Y(X_0,v) : v \in \partial u (x_0) \right\} \subset F_u(x_0).
\end{equation}
This will immediately imply $Y_0 \in F_u(x_0)$ and conclude the proof. Let us first observe that the above inclusion is equivalent to showing
\begin{equation}
\partial u(x_0) \subset \left\{D_x \phi(x_0, Y, X_0) : Y \in F_u(x_0) \right\} =: \mathcal{B}_{x_0}.
\end{equation}
To this end, we are going to show that the extremal points of $\partial u(x_0)$ are contained in $\mathcal{B}_{x_0}$ and that $\mathcal{B}_{x_0}$ is convex. The convexity of $\partial u(x_0)$ will then conclude the proof.

%To show that $\mathcal{B}_{x_0}$ contains the extremal points of $\partial u(x_0)$, 
Let $v_0 \in \partial u(x_0)$ be an extremal point. Then there exists a sequence $x_n \rightarrow x_0$ with $u$ differentiable at $x_n$ and $v_n = Du(x_n) \rightarrow v_0$. Let $X_n = (x_n, u(x_n))$, and let $Y_n \in F_u(x_n)$. Since $u$ is differentiable at $x_n$, it follows that $Y_n = Y(X_n, v_n)$. By compactness of $\Sigma$, we may assume $Y_n \rightarrow \overline{Y_0} \in \Sigma$. We claim $\overline{Y_0} = Y(X_0,v_0)$. Indeed, $u(x) \geq \phi(x, Y_n, X_n)$ for all $x \in \Omega$ with equality at $x = x_n$, so by letting $n \rightarrow \infty$, we obtain $u(x) \geq \phi(x, \overline{Y_0}, X_0)$ for all $x \in \Omega$, with equality at $x = x_0$. Note here that we are also using the continuity of $u$, since $X_n = (x_n, u(x_n)) \rightarrow (x_0, u(x_0)) =: X_0$. Furthermore, $v_n = Du(x_n) = D_x \phi(x_n, Y_n, X_n)$, so since $\phi$ is smooth (as a function of the variables $(x,Y,X)$) and $v_n \rightarrow v_0$, it follows that $v_0 = \lim\limits_{n \rightarrow \infty} v_n = \lim\limits_{n \rightarrow \infty} D_x \phi(x_n, Y_n, X_n) = D_x \phi(x_0, \overline{Y_0}, X_0)$. Combining all this, we conclude that $\overline{Y_0} \in F_u(x_0)$ and $D_x\phi(x_0, \overline{Y_0}, X_0) = v_0$, which shows $\overline{Y_0} = Y(X_0,v_0)$ and $v_0\in \mathcal{B}_{x_0}$.

On the other hand, to show that $\mathcal{B}_{x_0}$ is convex, let $Y_1, Y_2 \in F_u(x_0)$ and let $v_i = D_x \phi(x_0, Y_i, X_0)$ for $i = 1, 2$. Consider $v_{\lambda} = (1 - \lambda) v_1 + \lambda v_2$ and let $Y_{\lambda} = Y(X_0, v_{\lambda})$. Since $v_{\lambda} = D_x \phi(x_0, Y_{\lambda}, X_0)$, we have by the condition \eqref{eq:AWcondition} and Theorem \ref{maxprinciple} that for all $x \in \Omega$,
%\begin{center}
%$\displaystyle{
$$\phi(x, Y_{\lambda}, X_0) \leq \max \left\{\phi(x,Y_1, X_0), \phi(x,Y_2, X_0) \right\} \leq u(x).$$
%\end{center}
Hence $Y_{\lambda} \in F_u(x_0)$ and $\mathcal{B}_{x_0}$ is convex.

The proof is thus complete.
\end{proof}

Let us make another comparison with optimal mass transport. The set $\mathcal{B}_{x_0}$ is the analog of the $c$-subdifferential $\partial_c u(x_0)$ (cf. \cite{villani:stfleurBOOK}) in the theory of optimal mass transport. The inclusion $\partial_c u(x_0) \subset \partial u(x_0)$ is immediate from the definition of the $c$-subdifferential. The equality of the sets is obtained only after assuming the weak form of the condition (A3) and establishing Loeper's {\emph{DASM}} theorem. In the case of the parallel refractor, the inclusion $\mathcal{B}_{x_0} \subset \partial u(x_0)$ also follows from the definition of parallel refractor, as illustrated above in \eqref{TayExp}. The above analog of the {\emph{DASM}} theorem thus shows that under the condition \eqref{eq:AWcondition}, we have the equality $\partial u (x_0) = \mathcal{B}_{x_0}$ for all $x_0 \in \Omega$, which is the analog of the equality $\partial_c u(x_0) = \partial u(x_0)$ under the weak (A3) condition in the theory of optimal mass transport.

%Indeed, let $Y \in F_u(x_0)$ and consider $v = D_x\phi(x_0,Y,X_0)$, where $X_0 = (x_0, u(x_0))$. Then $v \in \mathcal{B}_{x_0}$ and $u(x) \geq \phi(x, Y, X_0)$ for all $x \in \Omega$ with equality if $x = x_0$. By Taylor expansion of $\phi$ in the $x$-variable, we have
%\begin{align*}
%u(x) & \geq \phi(x, Y, X_0) 
% = \phi(x_0, Y, X_0) + D_x\phi(x_0,Y,X_0)\cdot (x - x_0) + o(|x - x_0|) \\
%& = u(x_0) + v \cdot (x - x_0) + o(|x - x_0|),
%\end{align*}

%\noindent which shows $v \in \partial u(x_0)$ as desired.

%{\bf Remark:} The inclusion $\mathcal{B}_{x_0} \subset \partial u(x_0)$ is immediate from the definition of parallel refractor. Indeed, let $Y \in F_u(x_0)$ and consider $v = D_x\phi(x_0,Y,X_0)$, where $X_0 = (x_0, u(x_0))$. Then $v \in \mathcal{B}_{x_0}$ and $u(x) \geq \phi(x, Y, X_0)$ for all $x \in \Omega$ with equality if $x = x_0$. By Taylor expansion of $\phi$ in the $x$-variable, we have
%\begin{align*}
%u(x) & \geq \phi(x, Y, X_0) 
% = \phi(x_0, Y, X_0) + D_x\phi(x_0,Y,X_0)\cdot (x - x_0) + o(|x - x_0|) \\
%& = u(x_0) + v \cdot (x - x_0) + o(|x - x_0|),
%\end{align*}
%
%\noindent which shows $v \in \partial u(x_0)$ as desired. The above theorem thus shows that under the condition \eqref{eq:AWcondition}, we have the equality $\partial u (x_0) = \mathcal{B}_{x_0}$ for all $x_0 \in \Omega$.

\setcounter{equation}{0}
\section{Main results}\label{main}

\begin{Lemma}\label{cuno}
If $u$ is a parallel refractor, then there exists a structural positive constant $C$ such that
%\begin{center}
%$\displaystyle{
$$u((1 - s)\bar{x} + s \hat{x}) \leq (1 - s) u(\bar{x}) + s u(\hat{x}) + Cs(1-s)|\bar{x} - \hat{x}|^2$$
%\end{center}
for each $\bar{x}, \hat{x} \in \Omega$, and $s \in [0,1]$.
\end{Lemma}
\begin{proof} It can be proved following the arguments in \cite[Lemma 5.1]{gutierrez-tournier:regularityparallelrefractor} by exploiting \eqref{hessian} and the lower bound \eqref{lowerestimateforC}. The constant $C$ may be chosen dependent only on the bound for $D_x^2\phi(x,Y,X)$.
\end{proof}

The following lemma is crucial for the regularity of refractors, assuming $\Sigma$ regular from a point with respect to Definition \ref{regularityhypothesis}. We omit the proof since it can be completed proceeding as in \cite[Lemma 5.2]{gutierrez-tournier:regularityparallelrefractor}. The main needed ingredients for the proof are: the concavity of $\phi$, the estimates \eqref{lipschitzestimateforrefractor}-\eqref{mixedhessian}, Lemma \ref{unposu}, Lemma \ref{lemmafour}, and Lemma \ref{cuno}.

\begin{Lemma}\label{Taylor} Suppose $u$ is a parallel refractor and the target $\Sigma$ is regular from $X^* = (x^*, u(x^*))$. There exist positive constants $\delta, C_1, C_2$ depending on $X^*$ such that if $\bar{x}, \hat{x} \in B_{\delta}(x^*)$ and $\bar{Y} \in F_u(\bar{x})$, $\hat{Y} \in F_u(\hat{x})$, and $|\bar{Y} - \hat{Y}| \geq |\bar{x} - \hat{x}|$, then there exists $x_0 \in [\bar{x}, \hat{x}]$ such that if $X_0^* = (x_0, u(x_0))$, then
%\begin{center}
%$\displaystyle{
$$u(x) - \phi(x,Y,X_0^*) \geq -C|\bar{Y} - \hat{Y}||\bar{x} - \hat{x}| - C|Y(\lambda) - Y||x - x_0| + C_1|\bar{Y} - \hat{Y}|^2 |x - x_0|^2,$$
%\end{center}
for all $Y(\lambda) \in [\bar{Y}, \hat{Y}]_{X_0^*}$, $\lambda \in [\frac{1}{4}, \frac{3}{4}]$, $Y \in \Sigma$ and $x \in \Omega \cap B_{C_2}(x_0)$. The constant $C$ above depends only the derivatives bounds of $\phi$.
\end{Lemma}

The following theorem represents the first regularity-type result for refractors we show in this paper. Let us remark that our compatibility conditions \eqref{compatibilitycondition} play a key role in the proof.

Here and in what follows we denote by $\mathcal{N}_{\mu}(S)$ the $\mu$-neighborhood of a set $S$.

\begin{Theorem}\label{lowerboundtheorem} 
Suppose $u$ is a parallel refractor and $\Sigma$ is regular from $X^* = (x^*, u(x^*))$. Let $C, C_1, C_2, \delta$ be from the previous lemma. Then there exists a constant $M = M(C, C_1)$ such that if $\hat{x}, \bar{x} \in B_{\frac{\delta}{2}}(x^*)$, $\bar{Y} \in F_u(\bar{x})$ and $\hat{Y} \in F_u(\hat{x})$ satisfy 

\begin{equation}\label{lowerbound}
|\bar{Y} - \hat{Y}| \geq \max\left\{1, \left(\frac{2M}{\delta} \right)^2, \left(\frac{2M}{C_2} \right)^2 \right\}|\bar{x} - \hat{x}|
\end{equation}

\noindent then there exists $x_0 \in [\bar{x}, \hat{x}]$ such that

\begin{equation}\label{tube}
N_{\mu} \left(\left\{Y(\lambda) \in [\bar{Y}, \hat{Y}]_{X_0^*} : \lambda \in \left[\frac{1}{4}, \frac{3}{4} \right] \right\} \right) \cap \Sigma \subset F_u(B_{\eta}(x_0)),
\end{equation}

\noindent where $X_0^* = (x_0, u(x_0))$, $\mu = |\bar{Y} - \hat{Y}|^{\frac{3}{2}}|\bar{x} - \hat{x}|^{\frac{1}{2}}$, and $\eta = M \left(\frac{|\bar{x} - \hat{x}|}{|\bar{Y} - \hat{Y}|} \right)^{\frac{1}{2}}$.
\end{Theorem}

\begin{proof} Let $M= \dfrac{C + \sqrt{C^2 + 4C_1 C}}{C_1}$ and suppose $\bar{Y} \in F_u(\bar{x})$ and $\hat{Y} \in F_u(\hat{x})$ satisfy \eqref{lowerbound} with this choice of $M$. Define $\mu = |\bar{Y} - \hat{Y}|^{\frac{3}{2}}|\bar{x} - \hat{x}|^{\frac{1}{2}}$ and $\eta = \dfrac{M\mu}{|\bar{Y} - \hat{Y}|^2} =  M \left(\dfrac{|\bar{x} - \hat{x}|}{|\bar{Y} - \hat{Y}|} \right)^{\frac{1}{2}}$. Since $|\bar{Y} - \hat{Y}| \geq |\bar{x} - \hat{x}|$, the Lemma \ref{Taylor} applies. Let $x_0$ be the point in that lemma and let $X_0^* = (x_0, u(x_0))$. Fix $Y \in N_{\mu} \left(\left\{Y(\lambda) \in [\bar{Y}, \hat{Y}]_{X_0^*} : \lambda \in \left[\frac{1}{4}, \frac{3}{4} \right] \right\} \right) \cap \Sigma$. Then there exists $Y(\lambda) \in [\bar{Y}, \hat{Y}]_{X_0^*}$ with $\lambda \in [\frac{1}{4}, \frac{3}{4}]$ such that $|Y(\lambda) - Y| < \mu$. We already know we have the estimate 
$$u(x) - \phi(x,Y,X_0^*) \geq -C|\bar{Y} - \hat{Y}||\bar{x} - \hat{x}| - C\mu|x - x_0| + C_1|\bar{Y} - \hat{Y}|^2 |x - x_0|^2,$$
for all $|x-x_0| < C_2$. Observe that the right-hand side of the above expression is positive for all $x \in \Omega$ satisfying $|x - x_0| \geq \eta$ as long as $\eta$ satisfies the lower bound
%\begin{equation}\label{lowerboundoneta}
$$\eta > \frac{C\mu + \sqrt{C^2 \mu^2 + 4CC_1|\bar{Y} - \hat{Y}|^3|\bar{x} - \hat{x}|}}{2C_1|\bar{Y} - \hat{Y}|^2}.$$
%\end{equation}
Our choice of $M$ ensures that $\eta$ satisfies such inequality.% \eqref{lowerboundoneta}.

Next, observe that by \eqref{lowerbound}, $\eta \leq \frac{C_2}{2}$ and $\eta \leq \frac{\delta}{2}$. Since $x_0 \in B_{\frac{\delta}{2}}(x^*)$, it follows that $B_{\eta}(x_0) \subset B_{\delta}(x^*) \subset \Omega$, by the choice of $\delta$ in Lemma \ref{Taylor}. Thus, $u(x) > \phi(x, Y, X_0^*)$ for all $x \in \left(\Omega \cap B_{C_2}(x_0) \right) \backslash B_{\eta}(x_0)$.

We have to prove that $Y \in F_u(B_{\eta}(x_0))$. By definition of $X_0^*$, we have $u(x_0) = \phi(x_0, Y, X_0^*)$. Let $G_u$ denote the graph of $u$ on $B_{C_2}(x_0) \cap \Omega$, and consider $\sigma := \sup\left\{c(X,Y) - c(X_0^*, Y) : X \in G_u \right\}$, where $c(X,Y) = \kappa(y_{n+1} - x_{n+1}) - |X - Y|$. Since $X_0^* \in G_u$, we must have $\sigma \geq 0$. We claim $c(X,Y) - c(X_0^*, Y) < 0$ if $X \in G_u$ and $|x - x_0| \geq \eta$. For $x \notin B_{\eta}(x_0)$, let $X = (x, u(x))$ and $X' = (x, \phi(x,Y,X_0^*))$. Since $u(x) > \phi(x, Y, X_0^*)$ for $x \notin B_{\eta}(x_0)$, it follows that $X$ lies above $X'$. Hence, since $\partial_{x^0_{n+1}} c(X_0, Y) < 0$, it follows that $c(X,Y) < c(X', Y)$. On the other hand, $c(X',Y) = c(X_0^*, Y)$ by definition, and so the claim follows. Thus, the supremum $\sigma$ is attained at some $\tilde{X} = (\tilde{x}, u(\tilde{x})) \in G_u$ with $\tilde{x} \in B_{\eta}(x_0)$. To conclude the proof, we are going to show that $Y \in F_u(\tilde{x})$. 

To do this, we first show that, for all $x\in\Omega$,
\begin{equation}\label{consstrass}
\kappa(y_{n+1} - u(x)) \geq c(\tilde{X}, Y),\quad \mbox{and }\quad y_{n+1} - u(x) \geq \dfrac{\kappa}{\kappa^2 - 1} c(\tilde{X}, Y).
\end{equation}
Clearly, the second inequality implies the first. We want to explicitly remark that that the structural assumption \eqref{compatibilitycondition} allows us to obtain the second inequality. As a matter of fact, noticing that
$$c(\tilde{X},Y) = \kappa(y_{n+1} - \tilde{x}_{n+1}) - |\tilde{X} - Y| \leq (\kappa - 1)(y_{n+1} - \tilde{x}_{n+1}),$$
we have 
$$\frac{\kappa}{\kappa^2 - 1} c(\tilde{X}, Y) \leq \frac{\kappa}{\kappa + 1}(y_{n+1} - \tilde{x}_{n+1}).$$
We claim $\dfrac{\kappa}{\kappa + 1}(y_{n+1} - \tilde{x}_{n+1}) \leq y_{n+1} - u(x)$, which implies the second inequality in \eqref{consstrass}. Now, a simple rearrangement shows
$$\frac{\kappa}{\kappa + 1}(y_{n+1} - \tilde{x}_{n+1}) \leq y_{n+1} - u(x) \ \Leftrightarrow \ \kappa (u(x) - \tilde{x}_{n+1}) \leq y_{n+1} - u(x).$$
Recall that $\tilde{x}_{n+1} = u(\tilde{x})$. Since $\kappa > 1$ and $|x - \tilde{x}| \leq \Delta$, it follows from \eqref{lipschitzestimateforrefractor} that
$$|u(x) - \tilde{x}_{n+1}| = |u(x) - u(\tilde{x})| \leq \frac{|x - \tilde{x}|}{\sqrt{\kappa^2 - 1}} \leq \frac{\Delta}{\kappa - 1}.$$
By \eqref{compatibilitycondition} we have $\tau_1 - \tau_0 \geq \dfrac{\kappa \Delta}{\kappa - 1}$, which then implies
%\begin{equation}
$$\kappa(u(x) - \tilde{x}_{n+1}) \leq \dfrac{\kappa \Delta}{\kappa - 1} \leq \tau_1 - \tau_0 \leq y_{n+1} - u(x).$$
%\end{equation}
Therefore, the relations \eqref{consstrass} are satisfied. 

To conclude the proof, we have that $c(\tilde{X},Y) \geq c(X,Y)$ for all $X \in G_u$, that is,
%By definition of $c(X,Y)$, we thus obtain
%\begin{equation}
$$\sqrt{|x - y|^2 + (y_{n+1} - u(x))^2} \geq \kappa(y_{n+1} - u(x)) - c(\tilde{X}, Y)$$
%\end{equation}
for all $x \in \Omega \cap B_{C_2}(x_0)$. By \eqref{consstrass}, the right-hand side of the above inequality is non-negative, and so we get
\begin{align*}
& |x - y|^2 + (y_{n+1} - u(x))^2 \geq \kappa^2(y_{n+1} - u(x))^2 - 2\kappa (y_{n+1} - u(x)) c(\tilde{X}, Y) + c(\tilde{X},Y)^2\\
& \Leftrightarrow \frac{|x - y|^2}{\kappa^2 - 1} \geq (y_{n+1} - u(x))^2 - 2(y_{n+1} - u(x))\frac{\kappa}{\kappa^2 - 1}c(\tilde{X}, Y) + \frac{c(\tilde{X},Y)^2}{\kappa^2 - 1} \\
&\Leftrightarrow \frac{|x - y|^2}{\kappa^2 - 1} + \frac{c(\tilde{X},Y)^2}{(\kappa^2 - 1)^2}\geq \left(y_{n+1} - u(x) - \frac{\kappa}{\kappa^2 - 1}c(\tilde{X}, Y) \right)^2.
\end{align*}
Again by \eqref{consstrass} we infer
%\begin{equation}\label{localestimate}
$$u(x) \geq y_{n+1} - \frac{\kappa}{\kappa^2 - 1}c(\tilde{X},Y) - \sqrt{\frac{c(\tilde{X},Y)^2}{(\kappa^2 - 1)^2} + \frac{|x - y|^2}{\kappa^2 - 1}} = \phi(x, Y, \tilde{X})$$
%\end{equation}
for all $x \in \Omega \cap B_{C_2}(x_0)$, with equality at $x = \tilde{x}$. Since $\eta \leq \frac{C_2}{2}$ and $\tilde{x} \in B_{\eta}(x_0)$, we have $B_{\frac{C_2}{2}}(\tilde{x}) \subset B_{C_2}(x_0)$. Recall that by our choice of $\delta$, we have the inclusions $B_{\eta}(x_0) \subset B_{\delta}(x^*) \subset \Omega$ and so $B_{\epsilon}(\tilde{x}) \subset \Omega$ for all $\epsilon > 0$ sufficiently small. Hence, we obtain the local estimate $u(x) \geq \phi(x, Y, \tilde{X})$ for all $x \in B_{\epsilon}(\tilde{x})$ for $\epsilon > 0$ sufficiently small. By Theorem \ref{localtoglobal}, we obtain $u(x) \geq \phi(x, Y, \tilde{X})$ for all $x \in \Omega$, which shows $Y \in F_u(\tilde{x})$ and conclude the proof. \end{proof}

\subsection{H\"older regularity for refractors from growth conditions}\label{subsec:holderestimatesfromgrowthconditions}
%\subsection{$C^{1,\alpha}$ Regularity}
In this subsection we prove the $C^{1,\alpha}$-regularity result for refractors. We assume the regularity of the target from a point (Definition \ref{regularityhypothesis}), together with some growth conditions for the target measure which we are now going to introduce precisely.
\bigskip

%{\bf Growth Conditions on Target Measure:} 
Let $\sigma$ be a Radon measure on the target $\Sigma$ and let $X_0 \in \mathcal{C}_{\Omega}$. Assume there exists a neighborhood $\mathcal{U}_{X_0}$ and a constant $\tilde{C} > 0$ depending only on $X_0$ such that for all $\bar{Y}, \hat{Y} \in \Sigma$, $Z \in \mathcal{U}_{X_0}$ and $\mu > 0$ sufficiently small (depending on $X_0$), we have 
\begin{equation}\label{localcondition}
\sigma\left[\mathcal{N}_{\mu} \left(\left\{ Y(\lambda) \in [\bar{Y}, \hat{Y}]_Z : \lambda \in \left[\frac{1}{4}, \frac{3}{4} \right] \right\}\right) \cap \Sigma \right] \geq \tilde{C} \mu^{n-1} |\bar{Y} - \hat{Y}|.
\end{equation} 
%\noindent where $\mathcal{N}_{\mu}(S)$ denotes the $\mu$-neighborhood of a set $S$. 
We will also assume that the measure $\sigma$ satisfies 
\begin{equation}\label{growthcondition}
\sigma(F_u(B_{\eta})) \leq C_0 \eta^{\frac{n}{q}}
\end{equation} 
for all balls $B_{\eta} \subset \subset \Omega$ and some constant $C_0 > 0$ and $1 \leq q < \frac{n}{n-1}$.

\begin{Theorem}\label{thm:Holdercontinuityresult}
Suppose $u$ is a parallel refractor and the target $\Sigma$ is regular from $X^* = (x^*, u(x^*))$. Suppose the target measure $\sigma$ satisfies the growth conditions \eqref{localcondition} and \eqref{growthcondition} at $X^*$. Then there exists $\delta, M > 0$ and $C_2 > 0$ depending on $X^*$ such that if $\bar{x}, \hat{x} \in B_{\delta/2}(x^*)$, $\bar{Y} \in F_u(\bar{x})$ and $\hat{Y} \in F_u(\hat{x})$ satisfy
\begin{equation}\label{lowerboundagain}
|\bar{Y} - \hat{Y}| \geq \max\left\{1, \left(\frac{2M}{\delta} \right)^2, \left(\frac{2M}{C_2} \right)^2 \right\}|\bar{x} - \hat{x}|,
\end{equation}
then there exist positive constants $C_1 = C_1(C_0, \tilde{C}, M)$ and $\alpha = \alpha(n,q)$ such that 
%\begin{equation}\label{Holderestimate}
$$|\bar{Y} - \hat{Y}| \leq C_1 |\bar{x} - \hat{x}|^{\alpha}.$$
%\end{equation}
Moreover, 
%$$u\mbox{ is }C^{1, \alpha}\mbox{ in a neighborhood of }x^*.$$
$$u\in C^{1, \alpha}(B_{\delta/2}(x^*)).$$
The H\"older exponent can be obtained explicitly as
\begin{equation}\label{alphavalue}
\alpha = \frac{\frac{n}{2q} - \frac{n-1}{2}}{1 + \frac{3}{2}(n-1) + \frac{n}{2q}}.
\end{equation}
\end{Theorem}

\begin{proof} In Theorem \ref{lowerboundtheorem} we showed that under the assumption \eqref{lowerboundagain} (which is the same as \eqref{lowerbound} from Theorem \ref{lowerboundtheorem}) there exists $x_0 \in [\bar{x}, \hat{x}]$ such that if $X_0^* = (x_0, u(x_0))$, then we have the inclusion \eqref{tube}
$$N_{\mu} \left(\left\{Y(\lambda) \in [\bar{Y}, \hat{Y}]_{X_0^*} : \lambda \in \left[\frac{1}{4}, \frac{3}{4} \right] \right\} \right) \cap \Sigma \subset F_u(B_{\eta}(x_0)),$$
where $\mu = |\bar{Y} - \hat{Y}|^{\frac{3}{2}}|\bar{x} - \hat{x}|^{\frac{1}{2}}$, and $\eta = M \left(\frac{|\bar{x} - \hat{x}|}{|\bar{Y} - \hat{Y}|} \right)^{\frac{1}{2}}= \frac{M\mu}{|\bar{Y} - \hat{Y}|^2}$. It now follows from \eqref{localcondition} and \eqref{growthcondition} that
\begin{align*}
\tilde{C} \mu^{n-1} |\bar{Y} - \hat{Y}| & \leq \sigma\left[\mathcal{N}_{\mu} \left(\left\{ Y(\lambda) \in [\bar{Y}, \hat{Y}]_Z : \lambda \in \left[\frac{1}{4}, \frac{3}{4} \right] \right\}\right) \cap \Sigma \right] \\
& \leq \sigma[F_u(B_{\eta}(x_0))] \leq C_0 \eta^{\frac{n}{q}},
\end{align*}
which immediately implies $|\bar{Y} - \hat{Y}| \leq C_1 |\bar{x} - \hat{x}|^{\alpha}$. The value of $\alpha$ shown in \eqref{alphavalue} can be obtained using the definitions of $\mu$ and $\eta$. 

Moreover, it follows that $F_u(x)$ is single-valued for all $x \in B_{\delta/2}(x^*)$. Take $\bar{x} \in B_{\delta/2}(x^*)$. We first show that $u$ is differentiable at $\bar{x}$ and $D_x u(\bar{x}) = D_x \phi(\bar{x}, \bar{Y}, \bar{X})$, where $\bar{X} = (\bar{x}, u(\bar{x}))$ and $\bar{Y} = F_u(\bar{x})$. Indeed, for $i = 1, \ldots, n$ and $h$ sufficiently small, we obtain from \eqref{hessian} and the definition of parallel refractor that
\begin{align*}
D_i \phi(\bar{x}, \bar{Y}, \bar{X}) - \left( \frac{u(\bar{x} + h e_i) - u(\bar{x})}{h} \right) & \leq D_i \phi(\bar{x}, \bar{Y}, \bar{X}) - \left( \frac{\phi(\bar{x} + h e_i, \bar{Y}, \bar{X}) - \phi(\bar{x}, \bar{Y}, \bar{X})}{h} \right) \\
& = D_i \phi(\bar{x}, \bar{Y}, \bar{X}) - D_i \phi(\bar{x} + \tilde{h} e_i, \bar{Y}, \bar{X}) \\
& = -D_{ii} \phi(\bar{x} + \hat{h}e_i, \bar{Y}, \bar{X})\tilde{h}\leq C \tilde{h} \leq Ch
\end{align*}
for $0 \leq \hat{h} \leq \tilde{h} \leq h$. %In the final step, we have used the estimate \eqref{hessian} for the Hessian of $\phi$ to obtain the constant $C$. \\
For the reverse inequality, let $\hat{x} = \bar{x} + h e_i$, $\hat{X} = (\hat{x}, u(\hat{x}))$ and $\hat{Y} = F_u(\hat{x})$. Once again, by the definition of parallel refractor,
\begin{align*}
\left(\frac{u(\hat{x}) - u(\bar{x})}{h} \right) - D_i \phi(\bar{x}, \bar{Y}, \bar{X}) & \leq \left( \frac{\phi(\hat{x}, \hat{Y}, \hat{X}) - \phi(\bar{x}, \hat{X}, \hat{Y})}{h} \right) - D_i \phi(\bar{x}, \bar{Y}, \bar{X}) \\
& = D_i\phi(\tilde{x}, \hat{Y}, \hat{X}) - D_i \phi(\bar{x}, \bar{Y}, \bar{X})
\end{align*}
for some $\tilde{x} \in [\bar{x}, \hat{x}]$. By the first part of the proof, we have also $F_u(x)$ is a continuous function of $x$ and so $\hat{Y} \rightarrow \bar{Y}$ as $\hat{x} \rightarrow \bar{x}$. Thus, by letting $h \rightarrow 0$, we obtain the differentiability of $u$ and the desired formula for $Du$.

Finally, we show that $u \in C^{1, \alpha}(B_{\delta/2}(x^*))$. As a matter of fact, for $\bar{x}, \hat{x} \in B_{\delta/2}(x^*)$, we have
\begin{eqnarray*}
&&|D u(\bar{x}) - D u(\hat{x})| = |D_x \phi(\bar{x}, \bar{Y}, \bar{X}) - D_x \phi(\hat{x}, \hat{Y}, \hat{X})| \leq\\
&& \leq |D_x \phi(\bar{x}, \bar{Y}, \bar{X}) - D_x \phi(\hat{x}, \bar{Y}, \bar{X})| + |D_x \phi(\hat{x}, \bar{Y}, \bar{X}) - D_x \phi(\hat{x}, \hat{Y}, \bar{X})| +\\
&& + |D_x \phi(\hat{x}, \hat{Y}, \bar{X}) - D_x \phi(\hat{x}, \hat{Y}, \hat{X})| \leq\\
&& \leq C\left(|\bar{x} - \hat{x}| + |\bar{Y} - \hat{Y}| + |\bar{X} - \hat{X}| \right) \leq C |\bar{x} - \hat{x}|^{\alpha}.
\end{eqnarray*}
Here we have exploited the Lipschitz and the Hessian estimates for $\phi$, together with the facts that $|\bar{X} - \hat{X}| \leq C|\bar{x} - \hat{x}|$ and $|\bar{Y} - \hat{Y}| \leq C \max \left\{|\bar{x} - \hat{x}|, |\bar{x} - \hat{x}|^{\alpha} \right\}$.
\end{proof}

\setcounter{equation}{0}
\section{Appendix}\label{app}

The purpose of this appendix is to show existence of parallel refractors assuming an appropriate configuration and location of the target $\Sigma$. 
In fact, we will show existence of refractors for targets located within a slab of certain dimensions and so that \eqref{compatibilitycondition} holds, see Theorem \ref{thm:existencediscretecase} below. Therefore, the $C^{1,\alpha}$ estimates proved in the main body of this paper are applicable to actual refractors.
A main difference with the existence result in \cite{gutierrez-tournier:parallelrefractor} is that the geometry is now given by hyperboloids, then some estimates are different and require explanation. 
We show explicit estimates for the parameters in the configuration that yield existence of parallel refractors. In particular, we seek refractors satisfying the energy conservation condition. 

Let us first recall some notions. 
Given a refractor $u$ and a point $Y\in\Sigma$, 
the tracing mapping of $u$ is defined by
%\marginpar{changed}
$$\mathcal T_u(Y)=\left\{x\in\Omega\,:\,\text{$\H(Y,b)$ supports $u$ from below at $x$ for some $b>0$}\right\}.$$
If $E\subset\Sigma$, we put
$$\mathcal T_u(E)=\bigcup_{Y\in E}{\mathcal T_u(Y)}.$$
Given a nonnegative $f\in L^1(\Omega)$, and assuming the visibility condition \eqref{eq:visibilitycondition} below, the refractor mapping induces a Borel measure, the refractor measure, given by 
$$\mathcal M_{u}(E)=\int_{\mathcal T_u(E)}{f(x)\,dx};$$
this is proved as in \cite[comment after Definition 2.3]{gutierrez-tournier:parallelrefractor}.
%\marginpar{this paragraph\\and the one\\before\\are changed}
The function $f$ represents the intensity of the radiation emanating from $\Omega$. The radiation intensity to be received at $\Sigma$ is given by a Radon measure $\mu$. By assuming the energy conservation condition $\int_\Omega f(x)\,dx=\mu(\Omega)$, we will prove the existence of a parallel refractor $u:\Omega\to [0,\tau_0]$ (in the sense of Definition \ref{def:refractor}) for which the compatibility conditions \eqref{compatibilitycondition} are satisfied, and such that
$$\mathcal M_{u}(E)= \mu(E)\qquad \forall E\subseteq \Sigma.$$
The main step is to prove this in the discrete case, i.e., when $\mu=\sum_{i=1}^Na_i\delta_{Y_i}$ with $a_i>0$ and $Y_i\in\Sigma$. Once this is established, existence when $\mu$ is a general Radon measure follows by an approximation argument, see, e.g., \cite{{gutierrez-tournier:parallelrefractor}}.

\subsection{Geometric configuration of the target for existence of refractors}
Let us fix $Y_i=(y_i,y^i_{n+1})\in \Sigma$. Given $\mathbf b=(b_1,\cdots ,b_N)$ with $$0<b_i<\kappa\,y^i_{n+1}- \sqrt{(y^i_{n+1})^2+\Delta^2},\qquad 1\leq i\leq N,$$ 
we let
\[
u_{\mathbf b}(x)=\max_{1\leq i\leq N}\phi_{Y_i,b_i}(x)\qquad x\in \Omega.
\]
We will denote $M_{b}(Y_i)=M_{u_{\mathbf b}}(Y_i)$.

{\bf Step 1.}
We want to choose $0<\bar b_1\leq \kappa\,y^1_{n+1}- \sqrt{(y^1_{n+1})^2+\Delta^2}$ such that if $b_i=\kappa\,y^i_{n+1}- \sqrt{(y^i_{n+1})^2+\Delta^2}$ for $2\leq i\leq N$, then 
\[
\max_{2\leq i\leq N}\phi_{Y_i,b_i}(x)\leq
\phi_{Y_1,\bar{b_1}}(x)\qquad x\in \Omega;
\] 
that is, we will choose $\bar b_1$ such that 
\begin{equation}\label{eq:firstchoiceofbarb1}
\phi_{Y_1,\bar{b_1}}(x)\geq \phi_{Y_i,b_i}(x) \qquad \text{for all $x\in \Omega$ and $2\leq i\leq N$.} 
\end{equation}
We write 
\[
\varphi(s)=\dfrac{\kappa\,s}{\kappa^2-1}+\sqrt{\left( \dfrac{s}{\kappa^2-1}\right)^2+\dfrac{\Delta^2}{\kappa^2-1}}.
\]
We have 
\[
\min_{x\in \Omega}\phi_{Y_1,\bar{b_1}}(x)
=
y^1_{n+1}-\dfrac{\kappa\,\bar b_1}{\kappa^2-1}-\sqrt{\left(\dfrac{\bar b_1}{\kappa^2-1}\right)^2 + \dfrac{|x_1-y_1|^2}{\kappa^2-1}}
\]
for some $x_1\in \bar \Omega$. So
\[
\min_{x\in \Omega}\phi_{Y_1,\bar{b_1}}(x)
\geq 
y^1_{n+1}-\dfrac{\kappa\,\bar b_1}{\kappa^2-1}-\sqrt{\left(\dfrac{\bar b_1}{\kappa^2-1}\right)^2 + \dfrac{\Delta^2}{\kappa^2-1}}
=
y^1_{n+1}-\varphi\left(\bar b_1\right).
\]
On the other hand,
\begin{align*}
\phi_{Y_i,b_i}(x)&\leq \phi_{Y_i,b_i}(y_i)=y^i_{n+1}-\dfrac{b_i}{\kappa-1}\\
&=
y^i_{n+1}-\dfrac{1}{\kappa-1}\left(\kappa\,y^i_{n+1}- \sqrt{(y^i_{n+1})^2+\Delta^2} \right)\\
&=
\dfrac{1}{\kappa-1}\left(-y^i_{n+1}+ \sqrt{(y^i_{n+1})^2+\Delta^2} \right)
,\qquad 2\leq i\leq N;\forall x\in \Omega.
\end{align*}
So \eqref{eq:firstchoiceofbarb1} holds if we choose $\bar b_1>0$ such that 
\[
y^1_{n+1}-\varphi\left(\bar b_1\right)\geq \dfrac{1}{\kappa-1}\left(-y^i_{n+1}+ \sqrt{(y^i_{n+1})^2+\Delta^2} \right)\qquad 2\leq i\leq N, 
\]
so
\[
y^1_{n+1}-\varphi\left(\bar b_1\right)\geq \max_{2\leq i\leq N}\dfrac{1}{\kappa-1}\left(-y^i_{n+1}+ \sqrt{(y^i_{n+1})^2+\Delta^2} \right).
\]
This is equivalent to choose $\bar b_1$ such that
\begin{equation}\label{eq:inequalityforvarphibarb1}
\varphi\left(\bar b_1\right)\leq y^1_{n+1} +\min_{2\leq i \leq N}\left\{ \dfrac{1}{\kappa-1}\left(y^i_{n+1}- \sqrt{(y^i_{n+1})^2+\Delta^2} \right)\right\}:=y^1_{n+1} +m:=m^*;
\end{equation}
notice that $m<0$, and so $y^1_{n+1}>m^*$.
Now the function $\varphi$ is strictly increasing in $[0,+\infty)$ and $\varphi(0)=\dfrac{\Delta}{\sqrt{\kappa^2-1}}$, so to get some $\bar b_1>0$ satisfying \eqref{eq:inequalityforvarphibarb1}, we need that 
\begin{equation}\label{eq:lowerboundforM}
\dfrac{\Delta}{\sqrt{\kappa^2-1}}<m^*.
\end{equation}
It is easy to see by calculation that the inverse function of $\varphi$ is 
\[
\varphi^{-1}(y)=\kappa\,y-\sqrt{y^2+\Delta^2};\qquad \varphi^{-1}\left(\dfrac{\Delta}{\sqrt{\kappa^2-1}} \right)=0;
\]
$\varphi^{-1}:\left[\dfrac{\Delta}{\sqrt{\kappa^2-1}},+\infty\right)\to \R^+$.
Since $y^i_{n+1}\geq \tau_1$ for $1\leq i\leq N$, we have
\begin{align}
y^1_{n+1} +\dfrac{1}{\kappa-1}\left(y^i_{n+1}- \sqrt{(y^i_{n+1})^2+\Delta^2} \right)
&\geq
\tau_1 +\dfrac{1}{\kappa-1}\left(\tau_1- \sqrt{\tau_1^2+\Delta^2} \right)\notag\\
&=
\dfrac{1}{\kappa-1}\left( \kappa\, \tau_1 - \sqrt{\tau_1^2+\Delta^2} \right)=\dfrac{1}{\kappa-1}\varphi^{-1}(\tau_1)\label{eq:lowerestimateofM}\\
&>\dfrac{\Delta}{\sqrt{\kappa^2-1}}\quad\text{by choosing $\tau_1$ sufficiently large},\notag
\end{align}
so \eqref{eq:lowerboundforM} holds.
In addition, $\bar b_1$ must satisfy that $0<\bar b_1\leq \kappa\,y^1_{n+1}- \sqrt{(y^1_{n+1})^2+\Delta^2}$;
so we pick $\bar b_1$ as large as possible in this interval and satisfying \eqref{eq:inequalityforvarphibarb1} (notice that from \eqref{eq:lowerboundforM}, \eqref{eq:lowerboundforyn+1} always holds for all $b>0$ sufficiently small).
That is, we define
\begin{equation}\label{eq:definitionofbarb1}
\bar b_1=\max{ \left\{b:0<b\leq \kappa\,y^1_{n+1}- \sqrt{(y^1_{n+1})^2+\Delta^2}\text{ and }\varphi(b)\leq m^*\right\}},
\end{equation}
with $m^*$ given in \eqref{eq:inequalityforvarphibarb1}.

We then have 
\begin{align}\label{eq:formulaforb1bar}
 \bar b_1&=\min \left\{\varphi^{-1}(m^*),\kappa\,y^1_{n+1}- \sqrt{(y^1_{n+1})^2+\Delta^2}  \right\}\notag\\
 &=
\min \left\{\kappa\,m^*- \sqrt{(m^*)^2+\Delta^2} ,\kappa\,y^1_{n+1}- \sqrt{(y^1_{n+1})^2+\Delta^2}  \right\}\notag\\
&=\kappa\,m^*- \sqrt{(m^*)^2+\Delta^2}=\varphi^{-1}(m^*),
\end{align}
since $\varphi^{-1}$ is increasing and $y^1_{n+1}> m^*$.

With the $\bar b_1$ already chosen by \eqref{eq:definitionofbarb1}, let 
\begin{eqnarray}\label{def:setW}
W=\left\{(b_2,\cdots,b_N):\text{$0<b_i\leq \kappa\,y^i_{n+1}- \sqrt{(y^i_{n+1})^2+\Delta^2}$ with}\right.&& \\
\left.\mathcal M_{(\bar b_1,b_2,\cdots ,b_N)}(Y_i)\leq a_i \text{ for } 2\leq i \leq N \right\}.&& \nonumber
\end{eqnarray}
The inequality \eqref{eq:firstchoiceofbarb1} implies that the set $W\neq \emptyset$,
since $\mathcal M_{(\bar b_1,b_2,\cdots ,b_N)}(Y_i)=0$ for $b_i=\kappa\,y^i_{n+1}- \sqrt{(y^i_{n+1})^2+\Delta^2}$ for $2\leq i\leq N$.

{\bf Step 2}.
We prove that the vectors in the set $W$ are bounded below by a positive constant depending only on the constant $\tau_1$ in \eqref{eq:lowerboundforyn+1} (and \eqref{eq:lowerestimateofM}), the constant $\tau_2$ concerning the location of the target, $\Delta$, and $\kappa$. That is, we prove that if $(b_2,\cdots ,b_N)\in W$,  then $b_i\geq \delta>0$, for all $2\leq i\leq N$ with $\delta$ to be calculated; see \eqref{eq:definitionofdelta}.
Suppose that for some $(b_2,\cdots ,b_N)\in W$ there is $2\leq j_0\leq N$ such that $b_{j_0}<\delta$.
We shall prove that this implies that
\begin{equation}\label{eq:inequlitybetweenbj0andbarb1}
\phi_{Y_{j_0},b_{j_0}}(x)\geq \phi_{Y_1,\bar b_1}(x) \qquad \forall x\in \Omega,
\end{equation}
which implies that $\mathcal M_{(\bar b_1,b_2,\cdots ,b_N)}(Y_1)=0$ contradicting the energy conservation condition. We will prove that 
\begin{equation}\label{eq:maxj0biggerthanmaxbarb1}
\min_{x\in \Omega}\phi_{Y_{j_0},b_{j_0}}(x)\geq \max_{x\in \Omega}\phi_{Y_1,\bar b_1}(x)
\end{equation}
which clearly implies \eqref{eq:inequlitybetweenbj0andbarb1}.
We have for some $x_1\in \bar \Omega$
\begin{align*}
\min_{x\in \Omega}\phi_{Y_{j_0},b_{j_0}}(x)&=\phi_{Y_{j_0},b_{j_0}}(x_1)=
y^{j_0}_{n+1}-\dfrac{\kappa\,b_{j_0}}{\kappa^2-1}-\sqrt{\left(\dfrac{b_{j_0}}{\kappa^2-1}\right)^2 + \dfrac{|x_1-y_{j_0}|^2}{\kappa^2-1}}\\
&\geq 
y^{j_0}_{n+1}-\dfrac{\kappa\,b_{j_0}}{\kappa^2-1}-\sqrt{\left(\dfrac{b_{j_0}}{\kappa^2-1}\right)^2 + \dfrac{\Delta^2}{\kappa^2-1}}.
\end{align*}
On the other hand,
\[
\phi_{Y_1,\bar b_1}(y_1)=y^1_{n+1}-\dfrac{\bar b_1}{\kappa-1}\geq \max_{x\in \Omega}\phi_{Y_1,\bar b_1}(x).
\]
So to prove \eqref{eq:maxj0biggerthanmaxbarb1} it is enough to show that
\[
y^{j_0}_{n+1}-\dfrac{\kappa\,b_{j_0}}{\kappa^2-1}-\sqrt{\left(\dfrac{b_{j_0}}{\kappa^2-1}\right)^2 + \dfrac{\Delta^2}{\kappa^2-1}}
\geq
y^1_{n+1}-\dfrac{\bar b_1}{\kappa-1}
\]
if $b_{j_0}<\delta$.
Since $\varphi$ is strictly increasing, if we choose $\delta$ with 
\[
y^{j_0}_{n+1}-\varphi\left(\delta \right)
\geq
y^1_{n+1}-\dfrac{\bar b_1}{\kappa-1}
\]
then
\begin{equation}\label{eq:inequalityfordeltaandbj0}
y^{j_0}_{n+1}-\varphi\left(b_{j_0} \right)\geq y^{j_0}_{n+1}-\varphi\left(\delta \right)\geq
y^1_{n+1}-\dfrac{\bar b_1}{\kappa-1}.
\end{equation}
So we need to choose $\delta>0$ such that 
\[
\dfrac{\bar b_1}{\kappa-1}\geq y^1_{n+1}-y^{j_0}_{n+1}
+\varphi\left(\delta \right).
\]
We have $\varphi(0)=\dfrac{\Delta}{\sqrt{\kappa^2-1}}$, so to find $\delta>0$ satisfying the last inequality, we need to have 
\[
\dfrac{\bar b_1}{\kappa-1}-y^1_{n+1}+y^{j_0}_{n+1}>\dfrac{\Delta}{\sqrt{\kappa^2-1}}.\]
We have from \eqref{eq:formulaforb1bar}
\begin{align*}
&\dfrac{\bar b_1}{\kappa-1}-y^1_{n+1}+y^{j_0}_{n+1}=\dfrac{1}{\kappa-1}\,\varphi^{-1}(M)-y^1_{n+1}+y^{j_0}_{n+1}\\
&\geq 
\dfrac{1}{\kappa-1}\varphi^{-1}\left(\dfrac{1}{\kappa-1} \varphi^{-1}(\tau_1)\right)-y^1_{n+1}+y^{j_0}_{n+1}\qquad\text{(from \eqref{eq:lowerestimateofM} since $\varphi^{-1}$ is increasing)}\\
&\geq 
\dfrac{1}{\kappa-1}\varphi^{-1}\left(\dfrac{1}{\kappa-1} \varphi^{-1}(\tau_1)\right)-\tau_2+\tau_1\qquad \text{(since $\tau_1\leq y^j_{n+1}\leq \tau_2$ for $1\leq j\leq N$)}.
%&=
%\dfrac{\kappa}{\kappa-1} \,M-\dfrac{1}{\kappa-1}\,\sqrt{M^2+\Delta^2}-y^1_{n+1}+y^{j_0}_{n+1}\\
%&=
%\dfrac{\kappa}{\kappa-1} \,(y^1_{n+1}+m)-\dfrac{1}{\kappa-1}\,\sqrt{(y^1_{n+1}+m)^2+\Delta^2}-y^1_{n+1}+y^{j_0}_{n+1}\\
%&=
%\dfrac{1}{\kappa-1} \,y^1_{n+1}+\dfrac{\kappa}{\kappa-1}m-\dfrac{1}{\kappa-1}\,\sqrt{(y^1_{n+1}+m)^2+\Delta^2}+y^{j_0}_{n+1}\\
%&=
%\dfrac{1}{\kappa-1}\left( y^1_{n+1}-\sqrt{(y^1_{n+1})^2+\Delta^2}\right)+y^{j_0}_{n+1}\\
%&\geq \dfrac{1}{\kappa-1}\left( \tau_1-\sqrt{\tau_1^2+\Delta^2}\right)+\tau_1,
\end{align*}
%since the function $z-\sqrt{z^2+\Delta^2}$ is increasing and $y^j_{n+1}\geq \tau_1$ for $1\leq j\leq N$.
If we choose $\tau_1,\tau_2$ sufficiently large ($\tau_1$ satisfying also \eqref{eq:lowerboundforyn+1} and \eqref{eq:lowerestimateofM}) and satisfying
\begin{equation}\label{eq:conditionforC1andC2}
\dfrac{1}{\kappa-1}\varphi^{-1}\left(\dfrac{1}{\kappa-1} \varphi^{-1}(\tau_1)\right)-\tau_2+\tau_1>
\dfrac{\Delta}{\sqrt{\kappa^2-1}}
\end{equation} 
(notice that $w=\tau_2-\tau_1$ is the width of the slab containing the target, and \eqref{eq:conditionforC1andC2})
and choose $\delta$ with 
\[
\varphi(\delta)=\dfrac{1}{\kappa-1}\varphi^{-1}\left(\dfrac{1}{\kappa-1} \varphi^{-1}(\tau_1)\right)-\tau_2+\tau_1,
\]
then \eqref{eq:inequalityfordeltaandbj0} follows.
That is,
\begin{equation}\label{eq:definitionofdelta}
\delta=\varphi^{-1}\left(\dfrac{1}{\kappa-1}\varphi^{-1}\left(\dfrac{1}{\kappa-1} \varphi^{-1}(\tau_1)\right)-w \right).
\end{equation}

Therefore all $b$'s in $W$ are bounded below by $\delta$.

%Notice that $\Lambda(\tau_1):=\dfrac{1}{\kappa-1}\varphi^{-1}\left(\dfrac{1}{\kappa-1} \varphi^{-1}(\tau_1)\right)<\tau_1$, which follows iterating the inequality $\varphi^{-1}(\tau_1)<(\kappa-1)\,\tau_1$.
%This implies that $\delta<\varphi^{-1}(\tau_1)$.

In other words, given $w$, we can pick $\tau_1$ sufficiently large so that \eqref{eq:conditionforC1andC2} is satisfied.
We have then proved that 
\[
W\subset [\delta,\varphi^{-1}(\tau_1+w)]^{N-1}.
\]

{\bf Step 3}.
We shall prove the following upper bound for each refractor
\[
u_{\mathbf b}(x)=\max_{1\leq i\leq N}\phi_{Y_i,b_i}(x)\qquad x\in \Omega,
\]
when $b_1=\bar b_1$ from Step 1, and $(b_2,\cdots ,b_N)\in W$: 
\begin{equation}\label{eq:upperboundforub}
u_{\mathbf b}(x)\leq \tau_0:=\tau_1+w-\dfrac{\delta}{\kappa-1}
\end{equation}
with
\begin{equation}\label{eq:conditiononC1forregularity}
\tau_1\geq \max \left\{ \kappa \,\tau_0,\tau_0+\dfrac{\kappa\,\Delta}{\kappa-1}\right\},
\end{equation}
where $w=\tau_2-\tau_1$, for all $\tau_1$ is sufficiently large where $\delta$ is given in \eqref{eq:definitionofdelta}.

First notice that from \eqref{eq:maximumofphiYb}
\[
\phi_{Y_i,b_i}(x)\leq y_{n+1}^i-\dfrac{b_i}{\kappa-1}\leq \tau_2-\dfrac{\delta}{\kappa-1}
=
\tau_1+w-\dfrac{\delta}{\kappa-1}
\]
for $i=2,\cdots ,N$; and
\[
\phi_{Y_1,b_1}(x)\leq y_{n+1}^1-\dfrac{\bar b_1}{\kappa-1}\leq \tau_2-\dfrac{\bar b_1}{\kappa-1}.
\]
If we prove that 
\[
\bar b_1\geq \delta,
\]
then \eqref{eq:upperboundforub} holds.
In fact, from \eqref{eq:lowerestimateofM} and \eqref{eq:formulaforb1bar}
\[
\bar b_1=\varphi^{-1}(m^*)\geq \varphi^{-1}\left( \dfrac{1}{\kappa-1}\varphi^{-1}(\tau_1)\right):=L 
\]
since $\varphi^{-1}$ is increasing. 
Now
$\varphi(\delta)=\dfrac{1}{\kappa-1}\,L-w$, so $L\geq \delta$ is equivalent to 
$\varphi(\delta)-\dfrac{\delta}{\kappa-1}+w\geq 0$ which holds true since 
$\varphi(\delta)-\dfrac{\delta}{\kappa-1}>0$ and $w\geq 0$.
Notice that to obtain the bound \eqref{eq:upperboundforub} we must take $\tau_1$ sufficiently large satisfying the inequalities \eqref{eq:lowerboundforyn+1}, \eqref{eq:lowerestimateofM},
and \eqref{eq:conditionforC1andC2}.

We shall prove now that we can take $\tau_1$ even larger so that \eqref{eq:conditiononC1forregularity} holds.
We first show we can choose $\tau_1$ large so that 
\[
\tau_0\leq \dfrac{1}{\kappa}\,\tau_1.
\]
This inequality is equivalent to
\begin{equation}\label{eq:conditiononC1forregularitybis}
\dfrac{w}{\tau_1}+\left(1-\dfrac{1}{\kappa}\right)\leq \dfrac{\delta}{(\kappa-1)\,\tau_1}
=\dfrac{1}{(\kappa-1)\,\tau_1}\varphi^{-1}\left(\dfrac{1}{\kappa-1}\varphi^{-1}\left(\dfrac{1}{\kappa-1} \varphi^{-1}(\tau_1)\right)-w \right)=\mathcal L(\tau_1).
\end{equation}
Writing $\varphi^{-1}(\alpha)=\alpha\,\left(\kappa-\sqrt{1+\left(\dfrac{\Delta}{\alpha}\right)^2} \right)$, $\Lambda(\tau_1)=\dfrac{1}{\kappa-1}\varphi^{-1}\left(\dfrac{1}{\kappa-1} \varphi^{-1}(\tau_1)\right)$, and noticing that $\lim_{\tau_1\to +\infty}\varphi^{-1}(\tau_1)\to +\infty$ and $\lim_{\tau_1\to +\infty}\Lambda(\tau_1)-w\to +\infty$, we see that 
$\lim_{\tau_1\to +\infty}\mathcal L(\tau_1)=1.$
Hence \eqref{eq:conditiononC1forregularitybis} holds for all $\tau_1$ sufficiently large, since $\kappa>1$.
To show that we can choose $\tau_1$ large so that 
\[
\tau_1\geq \tau_0+\dfrac{\kappa\,\Delta}{\kappa-1},
\]
it is enough to show, since $\kappa>1$, that we can choose $\tau_1$ large such that
\[
\tau_0+\dfrac{\kappa\,\Delta}{\kappa-1}\leq \dfrac{1}{\kappa}\,\tau_1.
\]
The last inequality is equivalent to 
\[
\dfrac{w}{\tau_1}+\left(1-\dfrac{1}{\kappa}\right)+\dfrac{\kappa\,\Delta}{(\kappa-1)\,\tau_1}\leq \dfrac{\delta}{(\kappa-1)\,\tau_1}
=\mathcal L(\tau_1)
\]
which as before holds true for all $\tau_1$ sufficiently large.

In summary, we can choose $\tau_1$ large depending on $w$, $\kappa$ and $\Delta$ such that \eqref{eq:upperboundforub} and \eqref{eq:conditiononC1forregularity} hold true for any refractor with $b_1=\bar b_1$, and $(b_2,\cdots ,b_N)\in W$.
That is, the graph of the refractor is contained in the cylinder $\mathcal C_
\Omega=\Omega\times [0,\tau_0]$, with $\tau_0=\tau_1+w-\dfrac{\delta}{\kappa-1}$; where $\tau_1$ is large.

Recall once again that assuming the visibility condition \eqref{eq:visibilitycondition} below, $\mathcal M_{u}$ is a Borel measure in $\Sigma$
%\marginpar{this is changed\\compatible\\with\\the beginning}
as in \cite[comment after Definition 2.3]{gutierrez-tournier:parallelrefractor}. 
In addition, the continuity of the refractor measure follows as in \cite[Lemmas 2.4 and 3.2]{gutierrez-tournier:parallelrefractor} implying that the set $W$ in \eqref{def:setW} above is closed.
Then using the argument in the last third of the proof of \cite[Theorem 3.3]{gutierrez-tournier:parallelrefractor} yields the following existence theorem; see Figure \ref{fig:figurefarfield}.

\begin{figure}
  \centering
\includegraphics[width=5in]{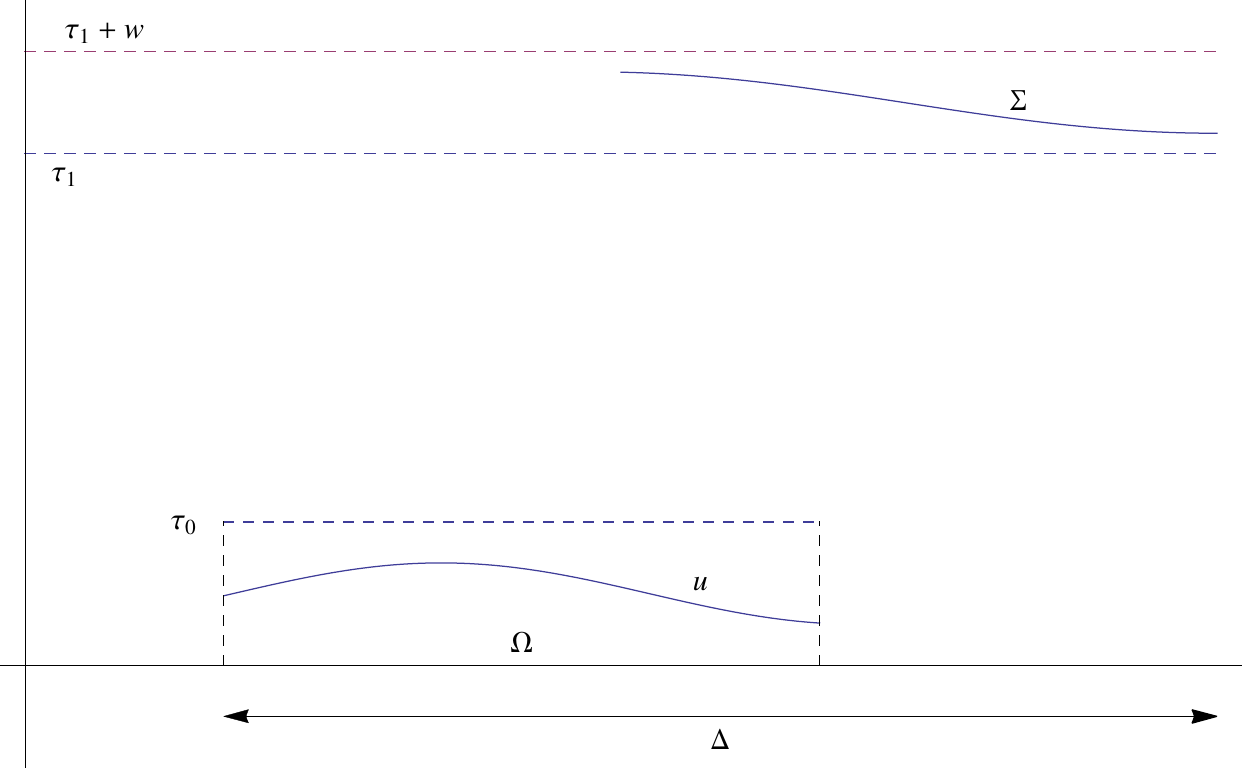}
\caption{Refraction configuration}
\label{fig:figurefarfield}
\end{figure}

\begin{Theorem}\label{thm:existencediscretecase}
Let $w>0$ be fixed. For all $\tau_1$ sufficiently large depending only on $w$, $\kappa$ and $\Delta$, there is $\tau_0$, given in \eqref{eq:upperboundforub} and satisfying \eqref{eq:conditiononC1forregularity}, such that for each target $\Sigma=\{Y_1,\cdots, Y_N\}$ contained in the slab $\{(y,y_{n+1}):\tau_1\leq y_{n+1}\leq \tau_1+w\}$ and satisfying the visibility condition:
\begin{align}\label{eq:visibilitycondition}
&\text{for all $X\in \Omega\times [0,\tau_0]$ and for all $m\in S^{n-1}$,}\\
&\text{the ray $\{X + tm : t > 0\}$ intersects $\Sigma$
in at most one point,}\notag
\end{align} 
there is a parallel refractor $u:\Omega\to [0,\tau_0]$ satisfying 
\begin{equation}\label{eq:energyconservation}
\int_{\mathcal T_u(E)}f(x)\,dx=\mu(E)\qquad \forall E\subset \Sigma,
\end{equation}
where $\mu=\sum_{i=1}^Na_i\delta_{Y_i}$ with $a_i>0$ and the energy conservation condition $\int_\Omega f(x)\,dx=\sum_{i=1}^Na_i$.
\end{Theorem}
 
Using the last theorem and proceeding as in the proof of the existence \cite[Theorem 3.4]{gutierrez-tournier:parallelrefractor},
we obtain by discretization that Theorem \ref{thm:existencediscretecase} holds true for an arbitrary Radon measure $\mu$ on a general target $\Sigma$ satisfying \eqref{eq:visibilitycondition} and the energy conservation condition $\int_\Omega f(x)\,dx=\mu(\Sigma)$.

%\bibliographystyle{amsalpha}
%\bibliography{monamp}
\providecommand{\bysame}{\leavevmode\hbox to3em{\hrulefill}\thinspace}
\providecommand{\MR}{\relax\ifhmode\unskip\space\fi MR }
\providecommand{\MRhref}[2]{%
  \href{http://www.ams.org/mathscinet-getitem?mr=#1}{#2}
}
\providecommand{\href}[2]{#2}

\end{document}